\documentclass[final]{siamltex}
 \pdfoutput=1
\usepackage{graphicx}
\usepackage{amsfonts}
\usepackage{mathrsfs}
\usepackage{epsfig}
\usepackage{amsmath}
\usepackage{amssymb}
\usepackage{url}
\usepackage{color}

\def\f{{\boldsymbol{f}}}
\def\g{{\boldsymbol{g}}}
\def\e{{\boldsymbol{e}}}
\def\u{{\boldsymbol{u}}}
\def\vv{{\boldsymbol{v}}}
\def\w{{\boldsymbol{w}}}
\def\xii{{\boldsymbol{\xi}}}
\def\muu{{\boldsymbol{\mu}}}

\def\E{{\boldsymbol{E}}}
\def\X{{\boldsymbol{X}}}
\def\H{{\boldsymbol{H}}}
\def\L{{\boldsymbol{L}}}

\newtheorem{remark}{\bf Remark}[section]
\usepackage{graphicx}
\usepackage{placeins}
\usepackage{cancel}
\usepackage{ulem}
\usepackage{marginnote}

\title{ Expandable Local and Parallel Two-Grid Finite Element Scheme for the Stokes Equations}
\author{Yanren Hou\thanks{School of Mathematics and Statistics, Xi'an Jiaotong University,
        Xi'an, Shaanxi 710049, China. ({\tt yrhou@mail.xjtu.edu.cn}). Partially supported by NSFC (Grant No. 11971378 \& 11571274).}
        \and Feng Shi\thanks{College of Science, Harbin Institute of Technology, Shenzhen 518055, China. ({\tt shi.feng@hit.edu.cn}).     Partially supported by  Foundation Research Project of Shenzhen (Grant No. JCYJ2018030617181319 \& ZDSYS201707280904031)}
        \and Haibiao Zheng \thanks{%
School of Mathematical Sciences, East China Normal University,
Shanghai Key Laboratory of Pure Mathematics and Mathematical
Practice, Shanghai, P.R. China. (\texttt{hbzheng13@gmail.com}).
Partially supported by   NSFC (Grant No. 11971174) and NSF
of Shanghai (Grant No. 19ZR1414300) and Science and Technology
Commission of Shanghai Municipality (Grant No. 18dz2271000 \& 19JC1420102).}
}
\begin{document}

\maketitle
\begin{abstract}
In this paper, we present a novel local and parallel two-grid finite element scheme for solving the Stokes equations,  
and rigorously establish its a priori error
estimates. The scheme admits simultaneously small scales of subproblems and distances between subdomains and its expansions, and hence can be expandable.
Based on the a priori error
estimates, we provide a corresponding iterative scheme with suitable iteration number. The resulting iterative  scheme
can reach the optimal convergence orders within specific two-grid
iterations ($O(|\ln H|^2)$ in 2-D and
$O(|\ln H|)$ in 3-D) if the coarse mesh size $H$ and the fine mesh size $h$
are properly chosen. Finally, some  numerical
tests including 2-D and 3-D cases are carried out to verify our theoretical results.
'
\end{abstract}



\section{Introduction}
Due to  the limiting of  computer resources, two-grid  finite element methods/nonlinear Galerkin schemes \cite{Xu1996,AMMI} and domain
decomposition methods \cite{CHAN} are popular and powerful tools for numerical simulations of linear and nonlinear PDEs nowadays. Such as  two-grid/two-level post-processing schemes for incompressible flow, we refer
\cite{ Layton1, HOULI1, HOULI2, LIHOU, LIUHOU,
XU} and the references therein for details.

In the past decades, a local and parallel two-grid finite element
method for elliptic boundary value problems was initially proposed
in \cite{XU1}. The scheme firstly solves the elliptic equation on a
coarse mesh  to get an initial lower frequency
guess of the solution. Then the whole computational domain
is divided into a series of disjoint subdomains { $\{D_j\}$} and the
driven term of the error equation, namely the residual term, is split into
several {  parts} defined only on {  such} small subdomains. Finally the global error equation can be
transferred into a series of subproblems with local driven terms.
Since the solution of  higher frequency  to each subproblem
decays very fast apart from the support of the local driven term, by suitablely
expanding each $D_j$ to the domain $\Omega_j\supset D_j$
and imposing the homogeneous boundary condition on $\partial\Omega_j$,
each subproblem can be approximated in a localized version defined
in the corresponding  { expanded} domain $\Omega_j$. The most attractive feature of the scheme is that
not any communication is required between local fine grid subproblems,
which makes the scheme a highly effective parallel scheme.
Such local and parallel two-grid scheme can be found  in \cite{Bank} and  has been extended to the
Stokes equations in \cite{HEXU1}. Error estimates derived in \cite{XU1,HEXU1} show that the approximate solutions in such schems can reach the optimal convergence orders in
both $H^1$ and $L^2$ norms.

However, according to \cite{NITSCHE},  the error
estimates are limited by the usage of the superapproximation
property of finite element spaces, which makes the error constant appeared in
\cite{XU1,HEXU1} to be a form of  $O(t^{-1})$,
where $t$ denotes the distance between $\partial D_j$ and
$\partial\Omega_j$. To obtain  the optimal error orders, usually $t=O(1)$ is required, which means the distance of { $\partial\Omega_j$}   { and $\partial D_j$} is almost a constant. Therefore one can not expect $\Omega_j$ to be arbitrary small. This  will prevent the corresponding scheme
from utilization in large parallel computer systems. In the previously mentioned local and parallel schemes,
 computational results  for  each $\Omega_j$ are usually removed
outside $D_j$, and are simply pasted  together to form the final
approximation, in which  discontinuity may appear along the boundaries of different
$D_j$. Instead, based on the method of partition of unity \cite{BABUSKA},  a local and parallel two-grid scheme is  proposed  for second order linear
elliptic equations \cite{WANG, ZHENG} and is also extended  to dealing with the Stokes equations and Navier-Stokes equations\cite{ZHENG2, ZHENG3}. Although the partition of unity method can guarantee that
the global approximation is continuous,  usage of  superposition principle  causes a crucial requirement that  the distance $t$ should be
   $O(1)$. To overcome such defects, some research on
linear elliptic equations has been done by the first author and his collaborator  in \cite{HOUDU} by
iterative method.

In this paper, based on the basic idea presented in \cite{XU1,HEXU1}, as an important extension of the idea in \cite{HOUDU},  we construct a local and parallel two-grid  iterative
scheme for solving the Stokes equations, in which the scale of each subproblem is
 $O(H)$, just requiring that $\mbox{diam}(D_j)=O(H)$ and $t=O(H)$ (much smaller than $O(1)$).
Since only a small overlapping in each two adjacent subproblems while $H$
tends to zero,   the scale of each subproblem can be arbitrary small as $O(H)$, that's a main reason why we
call the scheme an expandable local and parallel two-grid scheme. Meanwhile, in each cycle of two-grid iteration, to guarantee a better $L^2$ error estimate,  we  { adopt} a coarse grid correction.
Another main contribution
in this paper is that, to yield the globally
continuous { velocity}   in $\Omega$,  we use the principle of  superposition based on  { a} partition of unity to generate a series of local
and independent subproblems.  In particular, for patches of given
size,  to obtain a similar approximate accuracy as the one from the standard Galerkin
method on the fine mesh, we carry out rigorous analysis, and through the a priori
error estimate of the scheme, we  show that  a few number of iterations of order $O(|\ln H|^2)$ in 2-D and
$O(|\ln H|)$ in 3-D is only needed, respectively.  Similar technique has
been successfully applied for adaptive schemes with some a posterior error estimates in \cite{LARSON, LARSON1}.

The remainder of this paper is organized as follows. In Section 2, we will introduce the model problem and some preliminary materials.
 In Section 3, we will present our expandable local and parallel scheme for the Stokes equations. The   a priori
error estimates of the scheme are derived and then suggested iterative scheme is presented  in Section 4.
Some numerical experiments including 2-D and 3-D  examples are carried out in a parallel computer system  to support our
theoretical analysis in Section 5. Finally we give some conclusions in Section \sout{5} { 6}.

\section{Preliminaries}

Consider the following Stokes equations in a { smooth} bounded convex domain
$\Omega\subset R^d$, $d=2,3$.
\begin{equation}\label{Stokes}
\left\{\begin{array}{rl}
-\nu\Delta\mathbf{u}+\nabla p=\mathbf{f}, & \mbox{ in }\Omega,\\
\nabla\cdot\mathbf{u}=0, & \mbox{ in }\Omega,\\
\mathbf{u}=0, & \mbox{ on }\partial\Omega.
\end{array}\right.
\end{equation}
For a given bounded domain $D\subset R^d$,  standard notations for Sobolev spaces $W^{s,p}(D)$ and their
associated norms will be used, see, e.g., \cite{ADAMS,CIARLET}. In particular, for $p=2$, we simply define $H^s(D)=W^{s,2}(D)$,
$H_0^1(D)=\{v\in H^1(D): v|_{\partial D}=0\}$, and their norms as $\|\cdot\|_{s,D}=\|\cdot\|_{s,2,D}$ and $|\cdot|_{s,D}$.
We denote by $(\cdot,\cdot)_D$ the $L^2-$inner product on $D$. Therefore,  we have $\|\cdot\|_{0,D}=(\cdot,\cdot)_D^\frac12$, and  we use $\|\cdot\|_{D}$ to denote $\|\cdot\|_{0,D}$ in the rest of the paper.

Hereafter, we always use boldface characters to denote vector valued functions or spaces, for instance, $\H^s(D)=(H^s(D))^d$.
For simplicity, we will also use the symbols $\lesssim$, $\gtrsim$ and $\approxeq$ in the following sense: $x_1\lesssim y_1$, $x_2\gtrsim y_2$ and $x_3\approxeq y_3$ are equivalent to  $x_1\leq C_1y_1$, $x_2\geq c_2y_2$ and $c_3x_3\leq y_3\leq C_3x_3$ for some positive constants $C_1$, $c_2$, $c_3$ and $C_3$ independent of mesh
size.  We know that $\|\cdot\|_{1,D}\approxeq \|\nabla\cdot\|_{0,D}$ in $H^1_0(D)$. For subdomains $D_1\subset D_2\subset\Omega$, we use $D_1\subset\subset D_2$  to express that $\mbox{dist}(\partial D_1,\partial D_2)\stackrel{\triangle}{=}\mbox{dist}(\partial D_2\backslash\partial\Omega, \partial D_1\backslash\partial\Omega)>0$.

Furthermore, let us denote
\begin{eqnarray*}
&&\X(D)=\H^1(D),\quad\X_0(D)=\H_0^1(D),\quad M(D)=L^2(D)\\
&&Q(D)=L_0^2(D)=\{q\in L^2(D):\;\int_Dq=0\},
\end{eqnarray*}
and \sout{we} introduce the following bilinear forms
\begin{eqnarray*}
&&a_D(\u,\vv)=\nu(\nabla\u,\nabla\vv)_{D},\;\; d_D(p,\vv)=-(p,\nabla\cdot\vv)_{D}\;\;\forall \u,\vv \in\X(D),\;p\in M(D).
\end{eqnarray*}
By these notations, when $D=\Omega$, we simply denote by $\X=\X(\Omega)$, $\X_0=\X_0(\Omega)$, $M=M(\Omega)$, $Q=Q(\Omega)$, and $a(\cdot,\cdot)=a_\Omega(\cdot,\cdot)$, $d(\cdot,\cdot)=d_\Omega(\cdot,\cdot)$ and $(\cdot,\cdot)=(\cdot,\cdot)_\Omega$.
Then the weak formulation of (\ref{Stokes}) is: find $[\u,p]\in\X_0\times Q$ such that
$$
\left\{\begin{array}{ll}
a(\u,\vv)+d(p,\vv)=(\f,\vv)\quad &\forall\vv\in\X_0,\\
d(q,\u)=0 &\forall q\in Q.
\end{array}
\right.
$$
If we further introduce the bilinear form
$$
B_D([\u,p],[\vv,q])=a_D(\u,\vv)+d_D(p,\vv)-d_D(q,\u),
$$
and $B(\cdot,\cdot)=B_\Omega(\cdot,\cdot)$,
then the above weak form for the Stokes equations can be rewritten as
\begin{equation}\label{WStokes}
B([\u,p],[\vv,q])=(\f,\vv)\quad\forall [\vv,q]\in\X_0\times Q.
\end{equation}

Assume that $T^H(\Omega)=\{\tau_\Omega^H\}$ is a regular triangulation of $\Omega$, with the mesh size defined as
$H=\max_{\,\tau_\Omega^H\in T^H(\Omega)}\{\mbox{diam}(\tau_\Omega^H)\}$.
For certain given positive integer $r\geq 1$, let $\X^H\subset\X$ and $M^H\subset M$ be $C^0-$finite element spaces defined on $\Omega$ with approximation order $r+1$ in $\H^{r+1}(\Omega)$ and approximation order $r$ in $H^r(\Omega)$, respectively. We denote $\X_0^H=\X^H\cap\H^1_0(\Omega)$, $Q^H=M^H\cap L_0^2(\Omega)$ and assume $\X_0^H\times Q^H$ is a stable finite element pair.
Given $D\subset\Omega$, which aligns with $T^H(\Omega)$, we denote $T^H(D)$ as the restriction of $T^H(\Omega)$ on $D$, $\X^H(D)$ and $M^H(D)$ as the restriction of $\X^H$ and $M^H$ on $D$ and $Q^H(D)=M^H(D)\cap L_0^2(D)$.

Based on the finite element spaces defined above and weak formulation (\ref{WStokes}), we make the following assumptions.
\begin{itemize}
{ \item[{\bf A1.}] \emph{Interpolant}. Let $I_H$ be a Lagrange finite element interpolation of  $\X(D)$ onto $\X^H(D)$ and $\hat I_H=I-I_H$. There holds for any $\w\in\H^{s}(\tau_\Omega^H)$, $0\leq m\leq s$ and $s>d/2$,}
$$
\|\hat I_H\w\|_{m,\tau_\Omega^H}\lesssim H^{s-m}|\w|_{s,\tau_\Omega^H}.
$$
\item[{\bf A2}.] \emph{Inverse Inequality}. For any $\w\in\X^H(D)$, $0\leq m\leq s$, $1\leq p,q\leq\infty$,
$$\|\w\|_{s,p,\tau_\Omega^H}\lesssim H^{m-s+\frac{d}{p}-\frac{d}{q}}\|\w\|_{m,q,\tau_\Omega^H}.$$
\item[{\bf A3}.] \emph{Regularity}. Assume $D$ is smooth such that for any $\f\in\L^2(D)$, any solution to
$$
B_D([\u,p],[\vv,q])=(\f,\vv)_D\quad\forall [\vv,q]\in \X_0(D)\times Q(D),
$$
satisfies
$$
\|\u\|_{2,D}+\|p\|_{1,D}\lesssim \|\f\|_D.
$$
\end{itemize}

Now we can state the standard Galerkin approximation of (\ref{WStokes}) as follows: find $[\u_H,p_H]\in \X_0^H\times Q^H$ such that
\begin{equation}\label{SGM}
B([\u_H,p_H],[\vv_H,q_H])=(\f,\vv_H) \quad \forall\, [\vv_H,q_H]\in\X_0^H\times Q^H.
\end{equation}
For this problem, a  classical result holds, namely assuming that $[\u,p]\in (\H^{r+1}(\Omega)\times H^r(\Omega))\cap(\X_0\times Q)$, one have
\begin{equation}\label{ASGM}
\|\u-\u_H\|_{\Omega}+H(\|\nabla(\u-\u_H)\|_{\Omega}+\|p-p_H\|_\Omega)=O(H^{r+1}).
\end{equation}

\section{Local and Parallel Two-Grid Scheme}

In this section, we will introduce our local and parallel two-grid scheme for solving Stokes equations. Firstly, we define  the error functions by
$$
\hat\u=\u-\u_H\in\X_0,\quad \hat p=p-p_H\in Q,
$$
where $[\u,p]$ and $[\u_H, p_H]$ are solutions to (\ref{WStokes}) and (\ref{SGM}), respectively.
Then $[\hat\u,\hat p]$ satisfies the  error equation
\begin{equation}\label{errequ}
B([\hat\u,\hat
p],[\vv,q])=(\f,\vv)-B([\u_H,p_H],[\vv,q])\quad\forall[\vv,q]\in\X_0\times
Q.
\end{equation}

For any given
partition of unity of $\Omega$, namely $\{\phi_j\}_{j=1}^N$ with  $N\geq 1$ an integer,
$\Omega\subset\bigcup\limits_{j=1}^N\mbox{supp }\phi_j$ and
$\sum\limits_{j=1}^N\phi_j\equiv 1$ on $\Omega$, we denote $D_j=\mbox{supp}\,\phi_j$ and always assume
that $D_j$ aligns with $T^H(\Omega)$. With this patition of unity,
the error equation (\ref{errequ}) can be rewritten as:
\begin{equation*}
B([\hat\u,\hat
p],[\vv,q])=(\f,\sum\limits_{j=1}^N\phi_j\vv)-B([\u_H,p_H],\sum\limits_{j=1}^N\phi_j[\vv,q])\quad\forall[\vv,q]\in\X_0\times
Q.
\end{equation*}
Clearly, by the principle of superposition, this equation is equivalent
to the summation of the following $N$-subproblems:
\begin{equation}\label{errequ2}
B([\hat\u^j,\hat
p^j],[\vv,q])=(\f,\phi_j\vv)-B([\u_H,p_H],\phi_j[\vv,q])\quad\forall[\vv,q]\in
\X_0\times Q.
\end{equation}
Namely, $[\hat\u,\hat p]=\sum\limits_{j=1}^N[\hat\u^j,\hat p^j]$. The most important feature for such subproblems is that
all the
subproblems are mutually independent when $[\u_H,p_H]$ is known. Meanwhile
each subproblem is globally defined with homogeneous
Dirichlet boundary condition for the velocity, however is driven by
right-hand-side term with very small compact support, namely $D_j$.
As is
pointed out in \cite{LARSON1} (used also in \cite{HOUDU}) for elliptic problems, solution $\hat\u^j$ to each subproblem may decay very fast away from $D_j$.
So for the present Stokes problems, we will localize each subproblem to a small extension domain of $D_j$, namely $\Omega_j$, { and} impose { homogeneous Dirichlet} boundary condition for the velocity and discretize the subproblem on some triangulation $T^h(\Omega_j)$ for finite element approximation.

\begin{remark}\label{r_puf_dom}
For each $\Omega_j$, here we give a much natural choice.
The
partition of unity functions $\{\phi_j\}_{j=1}^N$ of $\Omega$ is selected as the piecewise linear Lagrange basis functions
associated with the coarse triangulation $T^H(\Omega)$, where
$N$ is the number of all vertices in $T^H(\Omega)$ including the
boundary ones. For each vertex $x_j$ of $T^H(\Omega)$, we can
denote $D_j=\mbox{supp }\phi_j$, and construct $\Omega_j$ by one layer extension of  $D_j$ in $T^H(\Omega)$, namely
$
\Omega_j=\bigcup_{x_i\in D_j} D_i.
$
We refer \cite{ZHENG2} for more details. Clearly, we know the scale esimates
\begin{equation}\label{3281411}
\mbox{diam}(D_j),\;\mbox{dist}(\partial
D_j,\partial\Omega_j)\approxeq H.
\end{equation}
\end{remark}

Generally, we need assume that each $D_j$ and its extension domain $\Omega_j$
are aligned with $T^H(\Omega)$. For each local domain $\Omega_j$, its regular triangulation is denoted by $T^h(\Omega_j)=\{\tau^h_{\Omega_j}\}$,
with mesh scale over all $j$ as $h=\max_{1\leq j\leq
N}\left\{\max_{\,\tau^h_{\Omega_j}\in
T^h(\Omega_j)}\{\mbox{diam}(\tau^h_{\Omega_j})\}\right\}$. As to the natural choices of partition of unity functions and subdomains as discussed in Remark \ref{r_puf_dom}, we can simply select local fine mesh $T^h(\Omega_j)$ as the limitation of the global regular triangulation of $\Omega$, $T^h(\Omega)=\{\tau^h_\Omega\}$, namely,  $T^h(\Omega_j)=T^h(\Omega)|_{\Omega_j}$, which will be used in the rest of this paper.

Base on the fine mesh triangulations $T^h(\Omega_j)$ and $T^h(\Omega)$ for each $\Omega_j$ and $\Omega$, we can introduce corresponding  finite element spaces $\X^{h}(\Omega_j)$, $\X_0^{h}(\Omega_j)$
and $\X^h$, $\X_0^h$, similarly defined as $\X^H$ and
$\X_0^H$ previously. We also introduce
$M^h(\Omega_j)$, $Q^h(\Omega_j)$, $M^h$ and $Q^h$ for the pressure accordingly. Noting that
any function in $\X_0^h(\Omega_j)$ can be extended to a function in
$\X_0^h$ through zero value  assignment over $\Omega\backslash\Omega_j$, we regard
$\X_0^h(\Omega_j)$ as a subspace of $\X_0^h$ in the sense of such
zero extension. In addition, we always assume that
\begin{equation}\label{hierarchical}\left\{\begin{array}{l}
\X^H\subset\X^h,\quad\X_0^H\subset\X_0^h=\bigcup\limits_{1\leq j\leq N}\X_0^h(\Omega_j),\\
M^H\subset M^h,\quad M^H|_{\Omega_j}\subset M^h(\Omega_j).
\end{array}\right.
\end{equation}
Besides, we define discontinuous finite
element spaces associated with the pressure
$$
M_j^h=\{q_h:\;q_h|_{\Omega_j}\in
M^h(\Omega_j),\;q_h|_{\Omega\backslash\Omega_j}\in
M^h(\Omega\backslash\Omega_j)\},\quad Q^h_j=M_j^h\cap L_0^2(\Omega).
$$
Clearly, $M^h\subset\bigcup\limits_{j=1}^NM_j^h\subset
M_H^h=\{q_h\in L^2(\Omega):q_h|_{\tau_H}\in M^h(\tau_H)\}$. Then for  $Q_H^h=M_H^h\cap L_0^2(\Omega)$, we have that $Q^h\subset
Q_H^h$ and the finite element  pair $\X_0^h\times Q_H^h$ is stable  (see
\cite{GIRRAV}).

Right now, we can present the localized finite element approximation to error equation (\ref{errequ2}) as follows:
find $[\hat\u^j_{H,h},\hat p^j_{H,h}]\in\X_0^{h}(\Omega_j)\times
Q^h(\Omega_j)$ such that $\forall [\vv_h,q_h]\in
\X_0^h(\Omega_j)\times Q^h(\Omega_j)$
\begin{equation}\label{alerrequ}
B_{\Omega_j}([\hat\u^j_{H,h},\hat
p^j_{H,h}],[\vv_h,q_h])=(\f,\phi_j\vv_h)_{\Omega_j}-B_{\Omega_j}([\u_H,p_H],\phi_j[\vv_h,q_h]),
\end{equation}
 in which $\hat\u^j_{H,h}$ and $\hat p^j_{H,h}$ are defined on $\Omega_j$. { In the rest, we always use the same symbols, that is $[\hat\u^j_{H,h},\hat p^j_{H,h}]$,  to denote their zero extension over $\Omega\backslash\Omega_j$. In such sense, $[\hat\u^j_{H,h},\hat p^j_{H,h}]\in \X_0^h\times Q_j^h(\Omega)$.}
Then we construct totally $$\hat\u_{H,h}=\sum\limits_{j=1}^N \hat\u^j_{H,h}\in\X_0^h,\quad
\hat p_{H,h}=\sum\limits_{j=1}^N \hat p_{H,h}^j\in Q_H^h,
$$
and the intermediate approximate solution
\begin{equation}\label{post1}
[\u_{H,h},p_{H,h}]=[\u_H+\hat\u_{H,h},p_H+\hat p_{H,h}].
\end{equation}
Since this approximation $[\u_{H,h},p_{H,h}]$ is derived by solving a
series of subproblems on { local} fine mesh   { with homogeneous boundary conditions}, its high frequency { error}
may be suppressed apparently. To balance the lower and higher frequency { errors},
a smooth step associated with the intermediate
approximation $[\u_{H,h},p_{H,h}]$ is resorted through coarse grid
correction as follows: find $[\E_\u^H, E_p^H]\in\X_0^H\times Q^H$ such that
$\forall[\vv_H,q_H]\in\X_0^H\times Q^H$
\begin{equation}\label{post2}
B([\E_u^H,E_p^H],[\vv_H,q_H])=(\f,\vv_H)-B([\u_{H,h},p_{H,h}],[\vv_H,q_H]).
\end{equation}
To this end, { an expected more accurate approximation than the coarse grid approximation is obtained,}
\begin{equation}\label{final}
[\u_H^h,p_H^h]=[\u_{H,h}+\E_\u^H,p_{H,h}+E_p^H]=[\u_H+\hat\u_{H,h}+\E_\u^H,p_H+\hat
p_{H,h}+E_p^H].
\end{equation}

In summary, we propose our local and parallel
two-grid scheme for solving (\ref{Stokes}) in the following.

\noindent{\bf{Local and parallel two-grid scheme:}}
\begin{description}
\item[Step 0.]  Deriving $[\u_H,p_H]$ by solving (\ref{SGM});
\item[Step 1.] Solving the equation (\ref{alerrequ})  to get $\{[\hat\u_{H,h}^j,\hat p_{H,h}^j]\}_{j=1}^N$ for each $j$, then constructing $[\u_{H,h},p_{H,h}]$ by formula (\ref{post1});
\item[Step 2.] Deriving $[\E_\u^h,E_p^h]$ by
solving (\ref{post2})  and finally constructing
$
[\u_H^h,p_H^h]
$
by (\ref{final}).
\end{description}

\section{Theoretical Analysis and  Suggested Local and Parallel Two-Grid Iterative Scheme}

As the basic step for analyzing the scheme above, we follow the idea discussed in \cite{HOUDU}, and extend the local sub-problem
(\ref{alerrequ}) to the original domain $\Omega$. For this purpose, we denote
$\Gamma=\partial\Omega$,
$\Gamma_j=\partial\Omega_j\backslash\Gamma$, and introduce a trace
space $\H^\frac12(\Gamma_j)=\H_0^1(\Omega)|_{\Gamma_j}$ on
$\Gamma_j$, which can be defined by interpolation (for example, see
\cite{JLLIONS})
$$
\H^\frac12(\Gamma_j)=\left\{
\begin{array}{ll}
[\L^2(\Gamma_j),\H^1(\Gamma_j)]_\frac12, \quad & \mbox{when}\,\Gamma_j\,\mbox{is a closed curve or surface},\\[0pt]
[\L^2(\Gamma_j),\H_0^1(\Gamma_j)]_\frac12, \quad &
\mbox{when}\,\Gamma_j\,\mbox{is a non-closed curve or surface}.
\end{array}
\right.
$$
We also denote by
$\H_h^\frac12(\Gamma_j)=\X_0^h(\Omega)|_{\Gamma_j}\subset\H^\frac12(\Gamma_j)$,
a finite dimensional trace space, equipped with the same
norm as in $\H^\frac12(\Gamma_j)$, and its dual space by
$\H_h^{-\frac12}(\Gamma_j)=(\H_h^\frac12(\Gamma_j))'$, which is
equipped with the norm
$$
\|\muu\|_{\H_h^{-\frac12}(\Gamma_j)}=\sup\limits_{\vv_h\in
\H_h^\frac12(\Gamma_j)}\frac{\int_{\Gamma_j}\vv_h\muu}{\|\vv_h\|_{\H^\frac12(\Gamma_j)}}.
$$

Then  based on the fictitious domain
method (see e.g., \cite{GLOWINSKI}), we can construct a problem with multiplier as
finding
$([\hat\u_{H,h}^j,\hat p_{H,h}^j],\xii^j)\in\X_0^h\times
Q_j^h\times\H_h^{-\frac12}(\Gamma_j)$, such that
\begin{eqnarray}\label{fictitious}
&&B([\hat\u_{H,h}^j,\hat p_{H,h}^j],[\vv_h,q_h])+<\xii^j,\vv_h>_{\Gamma_j}+<\muu,\hat\u_{H,h}^j>_{\Gamma_j}\\
&&\nonumber\qquad=(\f,\phi_j\vv_h)-B([\u_H,p_H],[\phi_j\vv_h,\phi_jq_h]),
\end{eqnarray}
holds for every $
([\vv_h,q_h],\muu)\in\X_0^h\times
Q_j^h\times\H_h^{-\frac12}(\Gamma_j)$, { where}
$$
<\muu,\vv_h>_{\Gamma_j}=\int_{\Gamma_j}\muu\vv_hds,\quad\forall
\muu\in\H_h^{-\frac12}(\Gamma_j),\;\vv_h\in \X_0^h.
$$
In the following, we will show the equivalence between problems (\ref{alerrequ}) and (\ref{fictitious}) in two steps.

Firstly, by introducing a subspace of $\X_0^h$
$$
\X_{j0}^h=\{\vv_h\in\X_0^h:\;<\muu,\vv_h>_{\Gamma_j}=0\;\;\forall
\muu\in\H_h^{-\frac12}(\Gamma_j)\}=\mbox{ker}(\gamma).
$$
 the above system (\ref{fictitious})
reduces to  finding $[\hat\u^j_{H,h},\hat p^j_{H,h}]\in\X_{j0}^h\times
Q_j^h$ such that
$$
B([\hat\u^j_{H,h},\hat
p^j_{H,h}],[\vv_h,q_h])=(\f,\phi_j\vv_h)-B([\u_H,p_H],[\phi_j\vv_h,\phi_jq_h]), \forall[\vv_h,q_h]\in\X_{j0}^h\times Q_j^h,
$$
which is obviously well-posed, since it  actually consists of two
independent Stokes problems defined in $\Omega_j$ and
$\Omega\backslash\Omega_j$ respectively, with homogeneous  boundary
conditions for { the} velocity.

Secondly,  for any given $\g\in\H_h^\frac12(\Gamma_j)$, we introduce two auxiliary
elliptic problems
$$
a_{\Omega_j}(\u_h^1,\vv_h)=0,\quad\u_h^1|_{\Gamma_j}=\g, \quad\u_h^1|_{\partial\Omega_j/\Gamma_j}=0, \quad\forall\vv_h\in\X_0^h(\Omega_j),
$$
and
$$
a_{\Omega\backslash\Omega_j}(\u_h^2,\vv_h)=0,\quad\u_h^2|_{\Gamma_j}=\g,\quad\u_h^2|_{\partial(\Omega/\Omega_j)/\Gamma_j}=0,\quad\forall\vv_h\in
\X_0^h(\Omega\backslash\Omega_j). $$
Clearly, these two problems establish two mappings $\gamma_1^{-1}$ and
$\gamma_2^{-1}$ from $\H_h^\frac12(\Gamma_j)$ into
$\X_E^h(\Omega_j)$ and $\X_E^h(\Omega\backslash\Omega_j)$,
respectively, with
\begin{eqnarray*}
    &&\X_E^h(\Omega_j)=\{\vv_h\in\X^h(\Omega_j): \vv_h|_{\partial\Omega_j\cap\partial\Omega}=0\},\\
    &&\X_E^h(\Omega\backslash\Omega_j)=\{\vv_h\in\X^h(\Omega\backslash\Omega_j): \vv_h|_{\partial\Omega\backslash\partial\Omega_j}=0\}.
\end{eqnarray*}
We can simply write as
$$
\u_h^1=\gamma_1^{-1}\g,\quad\u_h^2=\gamma_2^{-1}\g,
$$
which have the following estimates
$$
\|\gamma_1^{-1}\g\|_{\H^1(\Omega_j)},\;\|\gamma_2^{-1}\g\|_{H^1(\Omega\backslash\Omega_j)}\lesssim
\|\g\|_{\H^\frac12(\Gamma_j)}.
$$
We can also define a lifting operator $\gamma^{-1}$ from
$\H_h^\frac12(\Gamma_j)$ into $\X_0^h$: for any given
$\g\in\H_h^\frac12(\Gamma_j)$
$$
\gamma^{-1}\g=\left\{\begin{array}{ll}
\gamma_1^{-1}\g,\quad &\mbox{in}\;\Omega_j,\\
\gamma_2^{-1}\g,\quad &\mbox{in}\;\Omega\backslash\Omega_j,
\end{array}\right.
$$
which is the right inverse operator of the trace operator $\gamma$ from
$\X_0^h$ onto $\H_h^\frac12(\Gamma_j)$, with the property of
$$
\|\gamma^{-1}\g\|_{\H^1(\Omega)}\lesssim
\|\g\|_{\H^\frac12(\Gamma_j)}\quad\forall\g\in
\H_h^\frac12(\Gamma_j).
$$
Then we have for any $\muu\in\H_h^{-\frac12}(\Gamma_j)$
\begin{eqnarray}
    &&\|\muu\|_{\H_h^{-\frac12}(\Gamma_j)}=\sup\limits_{\g\in \H_h^\frac12(\Gamma_j)}\frac{<\muu,\g>_{\Gamma_j}}{\|\g\|_{\H^\frac12(\Gamma_j)}}\nonumber\\
    &&\qquad\lesssim
    \sup\limits_{\g\in \H_h^\frac12(\Gamma_j)}\frac{<\muu,\gamma^{-1}\g>_{\Gamma_j}}{\|\gamma^{-1}\g\|_{\H^1(\Omega)}}
    \leq\sup\limits_{\vv_h\in\X_0^h(\Omega)}\frac{<\muu,\vv_h>_{\Gamma_j}}{\|\vv_h\|_{\H^1(\Omega)}}, \label{inf-sup2}
\end{eqnarray}
which verifies that the bilinear form $<\cdot,\cdot>_{\Gamma_j}$ on
$\H_h^{-\frac12}(\Gamma_j)\times\X_0^h$ satisfies the inf-sup
condition.
Therefore, for any solution $[\hat\u^j_{H,h},\hat p^j_{H,h}]\in\X_0^h\times
Q_j^h$ to (\ref{alerrequ}), there exists a unique $\xii^j\in\H_h^{-\frac12}(\Gamma_j)$
such that $([\hat\u^j_{H,h},\hat p^j_{H,h}],\xii^j)$ satisfies the
 problem (\ref{fictitious}). { In such sense, the problems (\ref{alerrequ}) and (\ref{fictitious}) are equivalent.}

Now let us turn back to global and local problems
 (\ref{errequ}) and (\ref{errequ2}) related with the {  residuals of approximate solution}. 
 For the fine
triangulation $T^h(\Omega)$ and the associated finite element space $\X_0^h$,
their corresponding Galerkin approximations are finding
$[\hat\u_H,\hat p_H]\in\X_0^h\times Q^h$ and $[\hat\u_H^j,\hat
p_H^j]\in\X_0^h\times Q^h$, $j=1,2,\cdots,N$, such that
$\forall[\vv_h,q_h]\in\X_0^h\times Q^h$
\begin{equation}\label{errequh}
B([\hat\u_H,\hat
p_H],[\vv_h,q_h])=(\f,\vv_h)-B([\u_H,p_H],[\vv_h,q_h]),
\end{equation}
and
\begin{equation}\label{errequh2}
B([\hat\u_H^j,\hat
p_H^j],[\vv_h,q_h])=(\f,\phi_j\vv_h)-B([\u_H,p_H],[\phi_j\vv_h,\phi_jq_h]).
\end{equation}
We know that $[\hat\u_H,\hat
p_H]=[\sum\limits_{j=1}^N\hat\u_H^j,\sum\limits_{j=1}^N\hat p_H^j]$
is the Galerkin approximation of $[\hat\u,\hat p]$ in $\X_0^h\times
Q^h$, and meanwhile $[\u_h,p_h]=[\u_H,p_H]+[\hat\u_H,\hat p_H]\in\X_0^h\times
Q^h$ is the finite element approximation of
$[\u,p]$ on fine mesh.


Then if  denoting by
\begin{eqnarray*}
&&\e^j_{\u,h}=\hat\u^j_H-\hat\u^j_{H,h},\quad \e_{\u,h}=\sum\limits_{j=1}^N \e^j_{\u,h}=\u_h-\u_{H,h},\\
&&e^j_{p,h}=\hat p^j_H-\hat p^j_{H,h},\quad
e_{p,h}=\sum\limits_{j=1}^N e^j_{p,h}=p_h-p_{H,h},
\end{eqnarray*}
the local and the global error of $[\u_{H,h},p_{H,h}]$,
respectively,   { subtracting}
(\ref{errequ2}) from (\ref{fictitious}) gives
\begin{equation}\label{errorequj}
a(\e^j_{\u,h},\vv_h)+d(e_{p,h}^j,\vv_h)+<\xii^j,\vv_h>_{\Gamma_j}=0,\quad\forall
\vv_h\in\X_0^h,
\end{equation}
and further summing over all $j$'s yields
\begin{equation}\label{errorequ}
a(\e_{\u,h},\vv_h)+d(e_{p,h},\vv_h)+\sum\limits_{j=1}^N<\xii^j,\vv_h>_{\Gamma_j}=0,\quad\forall\vv_h\in
\X_0^h.
\end{equation}

In the following, we will carry out estimates on the local quantities $\e^j_{\u,h},e_{p,h}^j$ and $\xii^j$, which are very important in the error analysis of our scheme.
Firstly, for  $\xii^j$,
 by the previously defined operator $\gamma^{-1}$ and the fact that
 for every point $x\in \Omega$, there exists a
positive integer $\kappa$,
which is independent of  $N$ and $x$,
such that each $x$ belongs to at
most $\kappa$  different $\Omega_j$, we can easily derive the estimates as follows.
\begin{lemma}\label{LE0}{\it The multiplier $\xii^j$ in (\ref{fictitious}), also in (\ref{errorequj}), satisfies
$$
\|\xii^j\|_{\H_h^{-\frac12}(\Gamma_j)}\lesssim\|\nabla\e_{\u,h}^j\|_\Omega+\|e_{p,h}^j\|_\Omega,
$$
and
$$
\sum\limits_{j=1}^N<\xii^j,\vv_h>_{\Gamma_j}\lesssim
\kappa^\frac12\left(\sum\limits_{j=1}^N\|\xii^j\|^2_{\H_h^{-\frac12}(\Gamma_j)}\right)^\frac12\|\nabla\vv_h\|_\Omega\quad\forall\vv_h\in\X_0^h.
$$}
\end{lemma}

\begin{proof} The first estimate is a direct result of property (\ref{inf-sup2}).
For the second estimate, by the definition of
$\|\cdot\|_{\H_h^{-\frac12}(\Gamma_j)}$, we have
$\forall\vv_h\in\X_0^h$
\begin{eqnarray*}
&&\sum\limits_{j=1}^N<\xii^j,\vv_h>_{\Gamma_j}\leq\sum\limits_{j=1}^N\|\xii^j\|_{\H_h^{-\frac12}(\Gamma_j)}\|\vv_h\|_{\H^\frac12(\Gamma_j)}\lesssim\sum\limits_{j=1}^N\|\xii^j\|_{\H_h^{-\frac12}(\Gamma_j)}\|\vv_h\|_{\H^\frac12(\partial\Omega_j)}\\
&&\qquad\lesssim \sum\limits_{j=1}^N\|\xii^j\|_{\H_h^{-\frac12}(\Gamma_j)}\|\vv_h\|_{\H^1(\Omega_j)}\leq(\sum\limits_{j=1}^N\|\xii^j\|^2_{\H_h^{-\frac12}(\Gamma_j)})^\frac12(\sum\limits_{j=1}^N\|\vv_h\|^2_{\H^1(\Omega_j)})^\frac12\\
&&\qquad\leq\kappa^\frac12(\sum\limits_{j=1}^N\|\xii^j\|^2_{\H_h^{-\frac12}(\Gamma_j)})^\frac12\|\nabla\vv_h\|_\Omega.
\end{eqnarray*}
This is the second estimate of the lemma.
\end{proof}

As to the quantity of
$\|\nabla\e_{\u,h}^j\|_{\Omega}+\|e_{p,h}^j\|_\Omega$,
combining the identities
\begin{eqnarray*}
&&\|\nabla\e_{\u,h}^j\|_{\Omega}^2=\|\nabla(\hat\u_H^j-\hat\u_{H,h}^j)\|_{\Omega}^2=\|\nabla\hat \u_H^j\|_{\Omega\backslash\Omega_j}^2+\|\nabla(\hat\u_H^j-\hat\u_{H,h}^j)\|_{\Omega_j}^2,\\
&&\|e_{p,h}^j\|_{\Omega}^2=\|\hat p_H^j-\hat
p_{H,h}^j\|_{\Omega}^2=\|\hat
p_H^j\|_{\Omega\backslash\Omega_j}^2+\|\hat p_H^j-\hat
p_{H,h}^j\|_{\Omega_j}^2,
\end{eqnarray*}
and the equations satisfied by $[\e_{\u,h}^j|_{\Omega_j},e_{p,h}^j|_{\Omega_j}]$
\begin{equation*}
\left\{\begin{array}{ll}
a_{\Omega_j}(\e_{\u,h}^j|_{\Omega_j},\vv_h)+d_{\Omega_j}(e_{p,h}^j|_{\Omega_j},\vv_h)=0,\quad&\forall\vv_h\in \X_0^h(\Omega_j),\\
d_{\Omega_j}(q_h,\e^j_{\u,h}|_{\Omega_j})=0,\quad&\forall q_h\in Q^h(\Omega_j),\\
\e_{\u,h}^j|_{\partial\Omega_j}=\hat\u_H^j|_{\partial\Omega_j},
\end{array}\right.
\end{equation*}
 we have by virtue of  $\hat\u_H^j\in\X_0^h$
$$
\|\nabla\e_{\u,h}^j\|_{\Omega_j}+\|e_{p,h}^j\|_{\Omega_j}\lesssim
\|\hat \u_H^j\|_{\H^\frac12(\partial\Omega_j)}\leq\|\hat
\u_H^j\|_{\H^\frac12(\partial(\Omega\backslash\Omega_j))} { \leq}
\|\nabla\hat\u_H^j\|_{\Omega\backslash\Omega_j},
$$
which is arriving at
\begin{equation}\label{S1}
\|\nabla\e_{\u,h}^j\|_{\Omega_j}+\|e_{p,h}^j\|_{\Omega_j}\lesssim\|\nabla\hat\u_H^j\|_{\Omega\backslash\Omega_j}.
\end{equation}

Then we present the following two results related with the Sobolev space
$\H_0^1(\Omega)$ to make further analysis concerning on (\ref{S1}).
\begin{lemma}\label{LEH01}{\it
Let  ${\cal D}\subset\Omega$ with $\mbox{diam}({\cal D})\approxeq
H$ be any convex subdomain of $\Omega$. Then we have
$$
\|\w\|^2_{\L^2(\partial{\cal D})}\lesssim
H^{d-1}\beta_d(H)\|\nabla\w\|^2_{\Omega\backslash{\cal
D}}\quad\forall \w\in\H_0^1(\Omega),
$$
where $\beta_d(H)=H^{-1}$ when $d=3$ and $\beta_d(H)=|\ln H|$ when
$d=2$. }
\end{lemma}

\begin{proof}
Since ${\cal D}$ and $\Omega$ are convex domains,
there exists a point $P\in{\cal D}$ such that
$\mbox{dist}(P,\partial{\cal D})\approxeq H$ and
$\overline{PP'}\subset{\cal D}$ for any $P'\in\partial{\cal D}$ and
$\overline{PP'}\subset\Omega$ for any $P'\in\partial\Omega$. If we
denote $\Omega^c=\Omega\backslash{\cal D}$, we can establish a local
polar or spherical coordinate $(\rho,\omega)$ with origin at this
 point $P$ corresponding to 2-D to 3-D case,
see Fig.\ref{fig1} for the sketch of such setting in  the 2-D case.
\begin{figure}[ht]
   \begin{center}
        \includegraphics[width=0.6\linewidth]{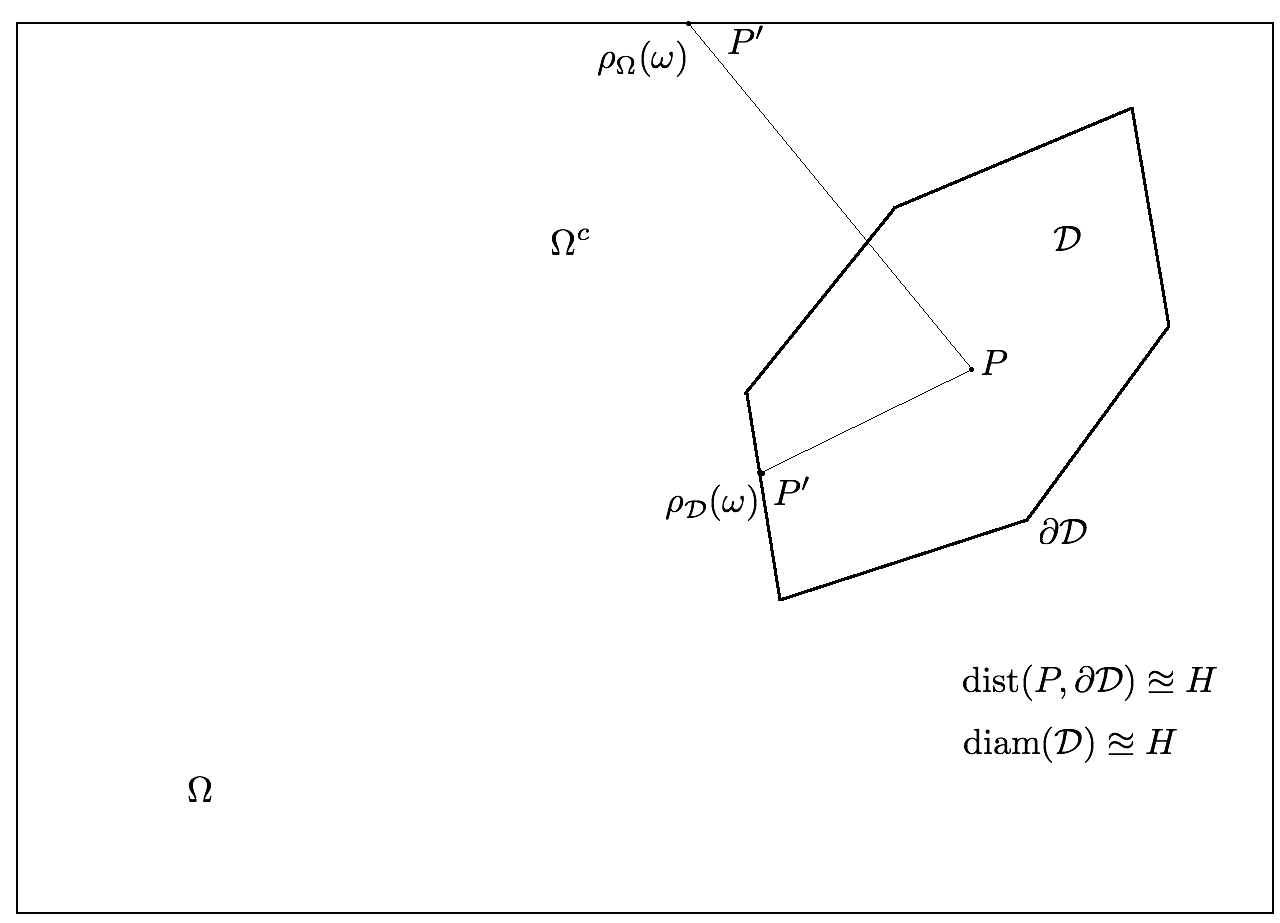}\\
        \caption{\small Local polar coordinate in 2-dimensional domain. \label{fig1}}
    \end{center}
\end{figure}
Here $\rho$ is the radius distant between any point $P'$ in $\Omega$ and the origin $P$, and $\omega=\omega_1$ in 2-D case and $\omega=(\omega_1,\omega_2)$
in 3-D case. The Jacobi determinant of the transformation between
the Cartesian coordinate $(x_1,\cdots, x_d)$ and $(\rho,\omega)$ is
$$
J=\frac{D(x_1,\cdots,x_d)}{D(\rho,\omega_1\cdots,\omega_{d-1})}=\rho^{d-1}\delta(\omega),\quad \mbox{with} \quad
\delta(\omega)=\sin^{d-2}\omega_1.
$$
Moreover, we use $\rho_\Omega(\omega)$ and $\rho_{\cal D}(\omega)$ to characterize
the boundary points of $\Omega$ and ${\cal D}$, respectively.

For any $\w\in\H_0^1(\Omega)$, since $\rho_{\cal D}(\omega)\approxeq
H$ and $H\lesssim\rho_\Omega(\omega)\lesssim 1$, and $\w(\rho_\Omega(\omega), \omega)=0$, we have
\begin{eqnarray*}
    &&|\w(\rho_{\cal D}(\omega),\omega)|=|-\int_{\rho_\Omega(\omega)}^{\rho_{\cal D}(\omega)}\frac{\partial \w}{\partial\rho}d\rho|\\
    &&\qquad\leq (\int^{\rho_\Omega(\omega)}_{\rho_{\cal D}(\omega)}\frac{1}{\rho^{d-1}}d\rho)^\frac12(\int^{\rho_\Omega(\omega)}_{\rho_{\cal D}(\omega)}\rho^{d-1}|\frac{\partial \w}{\partial\rho}|^2d\rho)^\frac12\\
    &&\qquad \lesssim \beta_d^\frac12(H)(\int^{\rho_\Omega(\omega)}_{\rho_{\cal D}(\omega)}\rho^{d-1}|\frac{\partial \w}{\partial\rho}|^2d\rho)^\frac12.
\end{eqnarray*}
Thus, if we denote by $\omega_d$ the unit ball in $R^d$, we have
\begin{eqnarray*}
    &&\|\w\|^2_{\L^2(\partial{\cal D})}=\int_{\partial\omega_d}(\rho_{\cal D}(\omega))^{d-1}| \w(\rho,\omega)|^2\delta(\omega)d\omega\\
    &&\qquad\lesssim \int_{\partial\omega_d}(\rho_{\cal D}(\omega))^{d-1}\beta_d(H)(\int^{\rho_\Omega(\omega)}_{\rho_{\cal D}(\omega)}\rho^{d-1}|\frac{\partial\w}{\partial\rho}|^2 d\rho)
    \delta(\omega)d\omega\\
    &&\qquad\lesssim H^{d-1}\beta_d(H)\|\nabla\w\|^2_{\Omega\backslash{\cal D}}.
\end{eqnarray*}
The proof is complete.
\end{proof}

\begin{lemma}\label{COR10241503}{\it
If ${\cal D}'\subset{\cal D}\subset\Omega$ are any two convex
subdomains of $\Omega$ with $\mbox{diam}({\cal D})\approxeq H$,
$\mbox{dist}(\partial{\cal D},\partial{\cal D}')\approxeq h$ and
$h<H$. Then we have
$$
\|\w\|^2_{{\cal D}\backslash{\cal D}'}\lesssim
hH^{d-1}\beta_d(H)\|\nabla\w\|^2_{\Omega\backslash{\cal
D}'}\quad\forall \w\in\H_0^1(\Omega).
$$
}
\end{lemma}

\begin{proof} Like what has been done in the proof of Lemma \ref{LEH01}, for ${\cal D}'\subset{\cal D}\subset\Omega$, we can choose some $P\in {\cal D}'$ such that $\mbox{dist}(P,\partial{\cal D}'),\mbox{dist}(P,\partial{\cal D})\approx H$, then we can establish a local
polar or spherical coordinate $(\rho,\omega)$ with origin at this
point $P$ corresponding to 2-D to 3-D case. And from the proof of Lemma \ref{LEH01}, we know for any point $(\rho,\omega)\in {\cal D}\backslash{\cal D}'$:
\begin{eqnarray*}
	&&|\w(\rho,\omega)|=|-\int_{\rho_\Omega(\omega)}^{\rho}\frac{\partial\w}{\partial\rho'}d\rho'|\leq (\int^{\rho_\Omega(\omega)}_{\rho}\frac{1}{\rho'^{d-1}}d\rho')^\frac12(\int^{\rho_\Omega(\omega)}_{\rho}\rho'^{d-1}|\frac{\partial\w}{\partial\rho'}|^2d\rho')^\frac12\\
	&&\qquad \lesssim \beta_d^\frac12(H)(\int^{\rho_\Omega(\omega)}_{\rho_{\cal D'}(\omega)}\rho'^{d-1}|\frac{\partial\w}{\partial\rho'}|^2d\rho')^\frac12.
\end{eqnarray*}
Then we have
\begin{eqnarray*}
	&&\|\w\|_{D\backslash D'}^2=\int_{\partial\omega_d}\int_{\rho_{D'(\omega)}}^{\rho_D(\omega)}\w^2\rho^{d-1}\delta(\omega)d\rho d\omega\\
	&&\quad\lesssim	\int_{\partial\omega_d}\int_{\rho_{D'(\omega)}}^{\rho_D(\omega)}\beta_d(H)[\int^{\rho_\Omega(\omega)}_{\rho_{\cal D'}(\omega)}\rho'^{d-1}|\frac{\partial \w}{\partial\rho'}|^2d\rho']\rho^{d-1}\delta(\omega)d\rho d\omega\\
	&&\quad\lesssim H^{d-1}\beta_d(H)\int_{\partial\omega_d}\int_{\rho_{D'(\omega)}}^{\rho_D(\omega)}[\int^{\rho_\Omega(\omega)}_{\rho_{\cal D'}(\omega)}\rho'^{d-1}|\frac{\partial \w}{\partial\rho'}|^2d\rho']\delta(\omega)d\rho d\omega\\
	&&\quad\leq H^{d-1}\beta_d(H)\int_{\partial\omega_d}(\int_{\rho_{D'(\omega)}}^{\rho_D(\omega)}[\int^{\rho_\Omega(\omega)}_{\rho_{\cal D'}(\omega)}\rho'^{d-1}|\frac{\partial \w}{\partial\rho'}|^2d\rho']^2\delta^2(\omega)d\rho)^\frac12(\int_{\rho_{D'(\omega)}}^{\rho_D(\omega)}d\rho)^\frac12 d\omega\\
	&&\quad\lesssim hH^{d-1}\beta_d(H)\int_{\partial\omega_d}\int^{\rho_\Omega(\omega)}_{\rho_{\cal D'}(\omega)}\rho'^{d-1}|\frac{\partial\w}{\partial\rho'}|^2d\rho'\delta(\omega)d\omega
	=hH^{d-1}\beta_d(H)\|\nabla\w\|^2_{\Omega\backslash D'},
\end{eqnarray*}
which completes the proof.
\end{proof}

Right now, we can give a more { rigorous} estimate for $\|\nabla
\e_{\u,h}^j\|_{\Omega}^2+\|e_{p,h}^j\|_\Omega^2$ based on  (\ref{S1}),
which will play a crucial role in the analysis of this section.

\begin{lemma}\label{LE2}
{\it Defining
$$
\alpha_d=\left\{\begin{array}{ll}
\frac{c}{|\ln H|^2}, \quad & d=2,\\
\frac{c}{|\ln H|}, & d=3,
\end{array}\right.
$$
where $c>0$ is a positive constant independent of $H$, $h$
and any subdomain $\Omega_j$. Then we have
\begin{eqnarray}\label{L44}
\|\nabla\e_{\u,h}^j\|_{\Omega}+\|e_{p,h}^j\|_\Omega\lesssim
H^{\alpha_d}(\|\nabla\hat\u_H^j\|_{\Omega}+\|\hat p_H^j\|_{\Omega}).
\end{eqnarray}
}
\end{lemma}

\begin{proof}{  The proof of this lemma is divided into four steps.

\emph{Step 1.} } Firstly, we use the same argument to split the region $\Omega_j\backslash
D_j$ for each $j$ as in \cite{HOUDU}.  We recall the sketch of a domain partition in Fig. \ref{MYFIG}.
\begin{center}
    \begin{figure}[ht]
        \centering
        \includegraphics[width=0.8\linewidth]{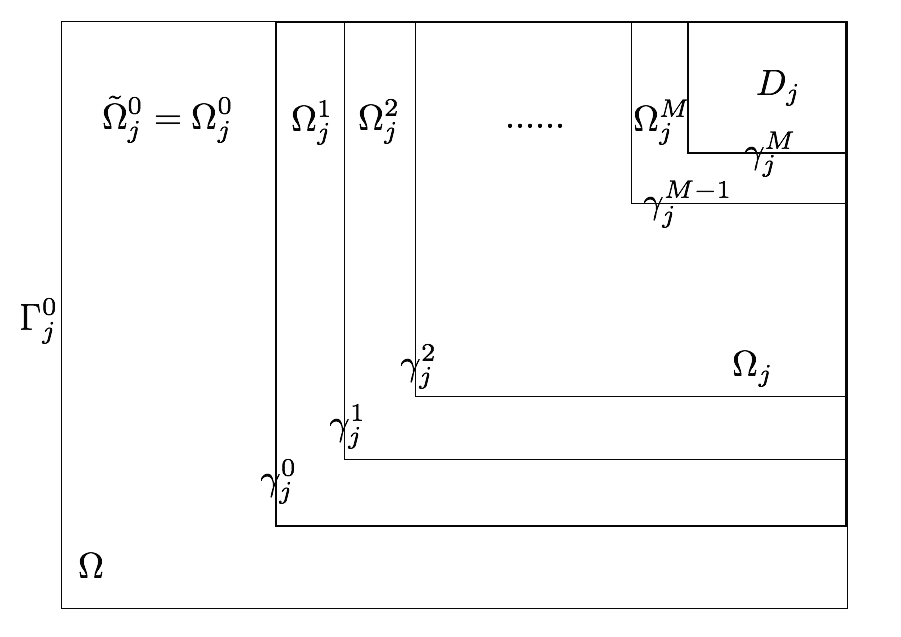}
        \caption{\small Division of the region $\Omega_j\backslash D_j$}\label{MYFIG}
    \end{figure}
\end{center}

Right now, we can define a series of disjoint annular zones as $\Omega_j^0,\Omega_j^1,\cdots,\Omega_j^{\cal M}$, and construct a sequence of subdomains as
$$
\tilde\Omega_j^k=\bigcup\limits_{i=0}^k\Omega_j^i,\quad
k=0,1,2,\cdots,{\cal M},
$$
with the property
$$
\tilde\Omega_j^0\subset\subset\tilde\Omega_j^1\subset\subset\cdots\subset\subset\tilde\Omega_j^{\cal
M}.
$$

Then we define
$$\gamma_j^k=\partial\tilde\Omega_j^k\backslash \partial\Omega,\quad\Gamma_j^k=\partial\tilde\Omega_j^k\backslash\gamma_j^k,\quad\mbox{and}\quad
\partial\tilde\Omega_j^k=\gamma_j^k\cup\Gamma_j^k,\quad\quad k=0,1,2,\cdots,{\cal M}.
$$

{  \emph{Step 2.} } {  Note that, since $\texttt{supp} \phi_j= D_j$, the right hand sides in (\ref{errequh2})
will be equal to 0 for the testing functions in $\X_0^h(\tilde\Omega_j^k)\times
M^h(\tilde\Omega_j^k)$ defined on the domain $\tilde\Omega_j^k\subset\subset \Omega\backslash D_j$}.
By denoting $M_0^h(\tilde\Omega_j^k)=M^h(\tilde\Omega_j^k)\cap
H^1_0(\tilde\Omega_j^k)$,
we clearly have
\begin{eqnarray}\label{L44-1}
B_{\tilde\Omega_j^k}([\hat\u_H^j,\hat
p_H^j],[\vv_h,q_h])=0\quad\forall[\vv_h,q_h]\in\X_0^h(\tilde\Omega_j^k)\times
M_0^h(\tilde\Omega_j^k).
\end{eqnarray}
{
For any $\theta\in R$, since $\vv_h\in \X_0^h(\tilde\Omega_j^k)$, we have
$$d_{\tilde\Omega_j^k}(\theta,\vv_h)=\int_{\tilde\Omega_j^k}\theta \nabla \cdot \vv_hdx=\theta \int_{\tilde\Omega_j^k}\nabla \cdot \vv_hdx=\theta \int_{\partial \tilde\Omega_j^k}  \vv_h\cdot n ds=0.$$
By the definition of $B$, we know
$$0=B_{\tilde\Omega_j^k}([\hat\u_H^j,\hat
p_H^j],[\vv_h,q_h])= B_{\tilde\Omega_j^k}([\hat\u_H^j,\hat
p_H^j],[\vv_h,q_h])-d_{\tilde\Omega_j^k}(\theta,\vv_h)=B_{\tilde\Omega_j^k}([\hat\u_H^j,\hat
p_H^j-\theta],[\vv_h,q_h]).$$
Then, we can split it into two parts since ${\tilde\Omega_j^k}=\tilde\Omega_j^{k-1}\cup\Omega_j^k$ as
$$0=B_{\tilde\Omega_j^k}([\hat\u_H^j,\hat
p_H^j-\theta],[\vv_h,q_h])=B_{\tilde\Omega_j^k-1}([\hat\u_H^j,\hat
p_H^j-\theta],[\vv_h,q_h])+B_{\Omega_j^k}([\hat\u_H^j,\hat
p_H^j-\theta],[\vv_h,q_h]),$$}
therefore, by simply using
\begin{eqnarray*}
&&\tilde\vv_h^{k}=\vv_h|_{\tilde\Omega_j^{k}},\quad\tilde\vv_h^{k-1}=\vv_h|_{\tilde\Omega_j^{k-1}},\quad\vv_h^k=\vv_h|_{\Omega_j^k},\\
&&\tilde q_h^{k}=q_h|_{\tilde\Omega_j^{k}},\quad\tilde q_h^{k-1}=q_h|_{\tilde\Omega_j^{k-1}},\quad
q_h^k=q_h|_{\Omega_j^k},
\end{eqnarray*}
we have for any $\theta\in R$
\begin{equation}\label{10291438}
B_{\tilde\Omega_j^{k-1}}([\hat\u_H^j,\hat
p_H^j-\theta],[\tilde\vv_h^{k-1},\tilde
q_h^{k-1}])=-B_{\Omega_j^k}([\hat\u_H^j,\hat
p_H^j-\theta],[\vv_h^k,q_h^k]).
\end{equation}

{  \emph{Step 3.}
 We firstly recall the following classical result (see for instance (2.11) in [24]):
    \begin{eqnarray*}
   ||q||_{0,\Omega}=\inf\limits_{c\in R}||q+c||_{0,\Omega}\ \ \ \forall q\in L^2_0(\Omega).
    \end{eqnarray*}
}
 Then, for convenience of expression, we denote by
$\X_E^h(\tilde\Omega_j^{k})=\X_0^h(\Omega \backslash D_j)|_{\tilde\Omega_j^{k}}$,
$M_E^h(\tilde\Omega_j^{k})=M_0^h(\Omega \backslash D_j)|_{\tilde\Omega_j^{k}}$, and
introduce
$$
\tilde\theta^{k}_{j,q_h}=\frac{1}{|\tilde\Omega_j^{k}|}\int_{\tilde\Omega_j^{k}}q_h,\quad
\theta^k_{j,q_h}=\frac{1}{|\Omega_j^k|}\int_{\Omega_j^k}q_h.
$$
Since
$\X_0^h(\tilde\Omega_j^{k})\subset\X_E^h(\tilde\Omega_j^{k})$,
{$M_E^h(\tilde\Omega_j^{k})\subset M^h(\tilde\Omega_j^{k})$} and
$\X_0^h(\tilde\Omega_j^{k})\times M^h(\tilde\Omega_j^{k})/R$ is
a stable finite element pair, we know
$\Sigma_j^{k}=\X_E^h(\tilde\Omega_j^{k})\times
M_E^h(\tilde\Omega_j^{k})/R$ is also a stable finite element pair,
that is
$$
\sup\limits_{\tilde\vv_h^{k}\in\X_E^h(\tilde\Omega_j^{k})}\frac{d_{\tilde\Omega_j^{k}}(\tilde
q_h^{k},\tilde\vv_h^{k})}{\|\nabla\tilde\vv_h^{k}\|_{\tilde\Omega_j^{k}}}\gtrsim\inf\limits_{\theta\in
R}\|\tilde q_h^{k}-\theta\|_{\tilde\Omega_j^{k}}=\|\tilde
q_h^{k}-\tilde\theta_{j,q_h}^{k}\|_{\tilde\Omega_j^{k}}.
$$
 We define a smooth function $\phi\in  { C^\infty(\tilde\Omega_j^k)}$ such that $\phi|_{\gamma_j^k}=0$ and 
$$
\mbox{supp}\,\phi=\tilde\Omega_j^k,\quad \phi(x)\equiv 1\;\forall x\in\tilde\Omega_j^{k-1},\quad
0\leq\phi\leq 1,\cancel{\quad\mbox{and}\quad| \phi(x)|\lesssim C}.
$$

{ In the following, we still use $I_h$ to denote the scalar valued Lagrange finite element interpolation, which share the same property in {\bf A1}.} Firstly, we can arrive at
\begin{eqnarray*}
&&\|\nabla\hat\u_H^j\|_{\tilde\Omega_j^{k}}+\|\hat
p_H^j-\tilde\theta_{j,\hat p_H^j}^{k}\|_{\tilde\Omega_j^{k}}=\|\nabla\hat\u_H^j\|_{\tilde\Omega_j^{k}}+\inf\limits_{\theta\in
R}\|\hat
p_H^j-\theta\|_{\tilde\Omega_j^{k}}\\
&&\lesssim\sup\limits_{[\tilde\vv_h^{k},\tilde q_h^{k}]\in\Sigma_j^{k}}\frac{B_{\tilde\Omega_j^{k}}([\hat\u_H^j,\hat p_H^j-\theta_{j,\hat p_H^j}^k],[\tilde\vv_h^{k},\tilde q_h^{k}])}{\|\nabla\tilde\vv_h^{k}\|_{\tilde\Omega_j^{k}}+\|\tilde q_h^{k}\|_{\tilde\Omega_j^{k}}}\\
&&\lesssim\sup\limits_{[\tilde\vv_h^{k},\tilde q_h^{k}]\in\Sigma_j^{k}}\frac{B_{\tilde\Omega_j^{k-1}}([\hat\u_H^j,\hat p_H^j-\theta_{j,\hat p_H^j}^k],[\tilde\vv_h^{k-1},\tilde q_h^{k-1}])+B_{\Omega_j^{k}}([\hat\u_H^j,\hat p_H^j-\theta_{j,\hat p_H^j}^k],[\vv_h^{k}, q_h^{k}])}{\|\nabla\tilde\vv_h^{k}\|_{\tilde\Omega_j^{k}}+\|\tilde q_h^{k}\|_{\tilde\Omega_j^{k}}}\\
&&\qquad=\sup\limits_{[\tilde\vv_h^{k},\tilde q_h^{k}]\in\Sigma_j^{k}}\frac{-B_{\Omega_j^{k}}([\hat\u_H^j,\hat p_H^j-\theta_{j,\hat p_H^j}^k],[I_h(\phi v_h^{k}), I_h(\phi q_h^{k})])+B_{\Omega_j^{k}}([\hat\u_H^j,\hat p_H^j-\theta_{j,\hat p_H^j}^k],[\vv_h^{k}, q_h^{k}])}{\|\nabla\tilde\vv_h^{k}\|_{\tilde\Omega_j^{k}}+\|\tilde q_h^{k}\|_{\tilde\Omega_j^{k}}}\\
&&\lesssim\sup\limits_{[\tilde\vv_h^{k},\tilde q_h^{k}]\in\Sigma_j^{k}}\frac{-B_{\Omega_j^{k}}([\hat\u_H^j,\hat p_H^j-\theta_{j,\hat p_H^j}^k],[I_h(\phi v_h^{k}), I_h(\phi q_h^{k})]) }{\|\nabla\tilde\vv_h^{k}\|_{\tilde\Omega_j^{k}}+\|\tilde q_h^{k}\|_{\tilde\Omega_j^{k}}}+\sup\limits_{[\tilde\vv_h^{k},\tilde q_h^{k}]\in\Sigma_j^{k}}\frac{ B_{\Omega_j^{k}}([\hat\u_H^j,\hat p_H^j-\theta_{j,\hat p_H^j}^k],[\vv_h^{k}, q_h^{k}])}{\|\nabla\tilde\vv_h^{k}\|_{\tilde\Omega_j^{k}}+\|\tilde q_h^{k}\|_{\tilde\Omega_j^{k}}}\\
&&:=I_1+ I_2.
\end{eqnarray*}
In the following, we will carry out further estimates for the two items.
\begin{eqnarray*}
I_1&&\qquad\lesssim\sup\limits_{{ \tilde\vv_h^{k}\in\X_E^h(\tilde\Omega_j^{k})}}\frac{|a_{\Omega_j^k}(\hat\u_H^j,I_h(\phi\vv_h^k))+d_{\Omega_j^k}(\hat p_H^j-\theta_{j,\hat p_H^j}^k,I_h(\phi\vv_h^k))|}{\|\nabla\tilde\vv_h^{k}\|_{\tilde\Omega_j^{k}}}\\
&&\qquad\quad+\sup\limits_{\tilde q_h^{k}\in
{ M_E^h(\tilde\Omega_j^{k})/R}}\frac{|d_{\Omega_j^k}(I_h(\phi q_h^k),\hat
u_H^j)|}{\|\tilde q_h^{k}\|_{\tilde\Omega_j^{k}}}\\
&&\qquad\lesssim { \sup\limits_{{ \tilde\vv_h^{k}\in\X_E^h(\tilde\Omega_j^{k})}}}\frac{ \|\nabla I_h(\phi\vv_h^k)\|_{\Omega_j^k}\|\nabla\hat\u_H^j\|_{\Omega_j^k}+\|\hat p_H^j-\theta_{j,\hat p_H^j}^k\|_{\Omega_j^k}\|\nabla\cdot(I_h(\phi\vv_h^k))\|_{\Omega_j^k}}{\|\nabla\tilde\vv_h^{k}\|_{\tilde\Omega_j^{k}}}\\
&&\qquad\quad+{ \sup\limits_{\tilde q_h^{k}\in
		M_E^h(\tilde\Omega_j^{k})}}\frac{\|I_h\phi q_h^k\|_{\Omega_j^k}\|\nabla\hat\u_H^j\|_{\Omega_j^k}}{\|\tilde q_h^{k}\|_{\tilde\Omega_j^{k}}}\\
&&\qquad\lesssim { \sup\limits_{{ \tilde\vv_h^{k}\in\X_E^h(\tilde\Omega_j^{k})}}}\frac{ h^{-1}\|\vv_h^k\|_{\Omega_j^k}\|\nabla\hat\u_H^j\|_{\Omega_j^k}+h^{-1}\|\hat p_H^j-\theta_{j,\hat p_H^j}^k\|_{\Omega_j^k}\|\vv_h^k\|_{\Omega_j^k}}{\|\nabla\tilde\vv_h^{k}\|_{\tilde\Omega_j^{k}}}\\
&&\qquad\quad+{ \sup\limits_{\tilde q_h^{k}\in
		M_E^h(\tilde\Omega_j^{k})}}\frac{\|q_h^k\|_{\Omega_j^k}\|\nabla\hat\u_H^j\|_{\Omega_j^k}}{\|\tilde q_h^{k}\|_{\tilde\Omega_j^{k}}},
\end{eqnarray*}
and
\begin{eqnarray*}
I_2&&\qquad\lesssim\sup\limits_{\tilde\vv_h^{k}\in\X_0(\Omega \backslash D_j)|_{\tilde\Omega_j^{k}}}\frac{|a_{\Omega_j^k}(\hat\u_H^j,\vv_h^k)+d_{\Omega_j^k}(\hat p_H^j-\theta_{j,\hat p_H^j}^k,\vv_h^k)|}{\|\nabla\tilde\vv_h^{k}\|_{\tilde\Omega_j^{k}}}
+\sup\limits_{\tilde q_h^{k}\in
M^h(\tilde\Omega_j^{k})}\frac{|d_{\Omega_j^k}(q_h^k,\hat
u_H^j)|}{\|\tilde q_h^{k}\|_{\tilde\Omega_j^{k}}}\\
&&\qquad\lesssim \frac{ h^{-1}\|\vv_h^k\|_{\Omega_j^k}\|\nabla\hat\u_H^j\|_{\Omega_j^k}+h^{-1}\|\hat p_H^j-\theta_{j,\hat p_H^j}^k\|_{\Omega_j^k}\|\vv_h^k\|_{\Omega_j^k}}{\|\nabla\tilde\vv_h^{k}\|_{\tilde\Omega_j^{k}}}
+\frac{\|q_h^k\|_{\Omega_j^k}\|\nabla\hat\u_H^j\|_{\Omega_j^k}}{\|\tilde q_h^{k}\|_{\tilde\Omega_j^{k}}}.
\end{eqnarray*}


{ We can choose ${\cal D}=\Omega\backslash\tilde\Omega_j^{k-1},\ {\cal D}\backslash{\cal D'}=\Omega_j^k,\ \Omega\backslash{\cal
D}'=\tilde\Omega_j^{k}$ in Lemma \ref{COR10241503}. Noting that $\mbox{diam}({\cal D})\approxeq H$,
$\mbox{dist}(\partial{\cal D},\partial{\cal D}')\approxeq h$ by the definitions of $\tilde\Omega_j^{k-1}$ and $\Omega_j^k$,} then
we can derive by virtue of  ({\bf A}2) and Lemma \ref{COR10241503}:
\begin{eqnarray*}
&&  h^{-1}\|\vv_h^k\|_{\Omega_j^k}\|\nabla\hat\u_H^j\|_{\Omega_j^k}\lesssim h^{-\frac12} H^\frac{d-1}{2}\beta_d^\frac12(H)\|\nabla\hat\u_H^j\|_{\Omega_j^k}\|\nabla\tilde\vv_h^{k}\|_{\tilde\Omega_j^{k}},\\
&& h^{-1}\|\hat p_H^j-\theta_{j,\hat p_H^j}^k\|_{\Omega_j^k}\|\vv_h^k\|_{\Omega_j^k}\lesssim h^{-\frac12}H^\frac{d-1}{2}\beta_d^\frac12(H)\|\hat
p_H^j-\theta_{j,\hat
p_H^j}^k\|_{\Omega_j^k}\|\nabla\tilde\vv_h^{k}\|_{\tilde\Omega_j^{k}}.
\end{eqnarray*}
{From $q_h\in M_0^h(\tilde\Omega_j^k)$, Lemma
\ref{COR10241503} and the inverse inequality, we also have}
{\begin{eqnarray*}
&& \|q_h^k\|_{\Omega_j^k}\|\nabla\hat\u_H^j\|_{\Omega_j^k}\lesssim h^\frac12H^\frac{d-1}{2}\beta_d^\frac12(H)\|\nabla\tilde q_h^{k}\|_{\tilde\Omega_j^{k}}\|\nabla\hat\u_H^j\|_{\Omega_j^k}\lesssim
h^{-\frac12}H^\frac{d-1}{2}\beta_d^\frac12(H)\|\tilde
q_h^{k}\|_{\tilde\Omega_j^{k}}\|\nabla\hat\u_H^j\|_{\Omega_j^k}.
\end{eqnarray*}}
Combining the estimates above leads to
\begin{eqnarray*}
&&\|\nabla\hat\u_H^j\|_{\tilde\Omega_j^{k}}+\inf\limits_{\theta\in R}\|\hat p_H^j-\theta\|_{\tilde\Omega_j^{k}}=\|\nabla\hat\u_H^j\|_{\tilde\Omega_j^{k}}+\|\hat p_H^j-\tilde\theta_{j,\hat p_H^j}^{k}\|_{\tilde\Omega_j^{k}}\\
&&\qquad\lesssim
h^{-\frac12}H^\frac{d-1}{2}\beta_d^\frac12(H)(\|\nabla\hat\u_H^j\|_{\Omega_j^k}+\|\hat
p_H^j-\theta_{j,\hat p_H^j}^k\|_{\Omega_j^k}).
\end{eqnarray*}

  Therefore we have
$$
\|\nabla\hat\u_H^j\|_{\tilde\Omega_j^{k}}^2+\|\hat
p_H^j-\tilde\theta_{j,\hat
p_H^j}^{k}\|_{\tilde\Omega_j^{k}}^2\lesssim
h^{-1}H^{d-1}\beta_d(H)(\|\nabla\hat\u_H^j\|_{\Omega_j^k}^2+\|\hat
p_H^j-\theta_{j,\hat p_H^j}^k\|_{\Omega_j^k}^2).
$$
{  It is obvious that
$$
\|\nabla\hat\u_H^j\|_{\tilde\Omega_j^{k-1}}^2+\|\hat
p_H^j-\tilde\theta_{j,\hat
	p_H^j}^{k-1}\|_{\tilde\Omega_j^{k-1}}^2\lesssim
h^{-1}H^{d-1}\beta_d(H)(\|\nabla\hat\u_H^j\|_{\Omega_j^k}^2+\|\hat
p_H^j-\theta_{j,\hat p_H^j}^k\|_{\Omega_j^k}^2),
$$	
}
or equivalently, there exists a mesh independent constant
$c>0$ so that
\begin{eqnarray}\label{L44-3}
\|\nabla\hat\u_H^j\|_{\Omega_j^k}^2+\|\hat p_H^j-\theta_{j,\hat
p_H^j}^k\|_{\Omega_j^k}^2\geq
{ chH^{1-d}\beta_d^{-1}(H)(\|\nabla\hat\u_H^j\|_{\tilde\Omega_j^{k-1}}^2+\|\hat
p_H^j-\tilde\theta_{j,\hat p_H^j}^{k-1}\|_{\tilde\Omega_j^{k-1}}^2)}.
\end{eqnarray}
{  \emph{Step 4.}}
By using the above inequality successively, we can deduce that
\begin{eqnarray*}
&&\|\nabla\hat\u_H^j\|_{\tilde\Omega_j^{\cal M}}^2+\|\hat p_H^j\|_{\tilde\Omega_j^{\cal M}}^2\\
&&\qquad\geq(\|\nabla\hat\u_H^j\|_{\tilde\Omega_j^{{\cal M}-1}}^2+\|\hat p_H^j-\tilde\theta_{j,\hat p_H^j}^{{\cal M}-1}\|_{\tilde\Omega_j^{{\cal M}-1}}^2)+(\|\nabla\hat\u_H^j\|_{\Omega_j^{\cal M}}^2+\|\hat p_H^j-\theta_{j,\hat p_H^j}^{\cal M}\|_{\Omega_j^{\cal M}}^2)\\
&&\qquad\geq (1+chH^{1-d}\beta_d^{-1}(H))(\|\nabla\hat\u_H^j\|_{\tilde\Omega_j^{{\cal M}-1}}^2+\|\hat p_H^j-\tilde\theta_{j,\hat p_H^j}^{{\cal M}-1}\|_{\tilde\Omega_j^{{\cal M}-1}}^2)\geq\cdots\\
&&\qquad\geq(1+chH^{1-d}\beta_d^{-1}(H))^{\cal M}(\|\nabla\hat\u_H^j\|_{\tilde\Omega_j^0}^2+\|\hat p_H^j-\tilde\theta_{j,\hat p_H^j}^0\|_{\tilde\Omega_j^0}^2)\\
&&\qquad=(1+chH^{1-d}\beta_d^{-1}(H))^{\cal
M}(\|\nabla\hat\u_H^j\|_{\Omega\backslash\Omega_j}^2+\|\hat
p_H^j-\tilde\theta_{j,\hat p_H^j}^0\|_{\Omega\backslash\Omega_j}^2).
\end{eqnarray*}
Noting that
$$
\|\nabla\hat\u_H^j\|_{\tilde\Omega_j^{\cal M}}^2+\|\hat
p_H^j\|_{\tilde\Omega_j^{\cal
M}}^2=\|\nabla\hat\u_H^j\|_{\Omega\backslash D_j}^2+\|\hat
p_H^j\|^2_{\Omega\backslash D_j}\leq
\|\nabla\hat\u_H^j\|_\Omega^2+\|\hat p_H^j\|^2_\Omega,
$$
we can get
$$
\|\nabla\hat\u_H^j\|_{\Omega\backslash\Omega_j}^2+\|\hat
p_H^j-\tilde\theta_{j,\hat
p_H^j}^0\|_{\Omega\backslash\Omega_j}^2\leq
(1+chH^{1-d}\beta_d^{-1}(H))^{-{\cal
M}}(\|\nabla\hat\u_H^j\|_{\Omega}^2+\|\hat p_H^j\|^2_{\Omega}).
$$

Finally, due to (\ref{S1}), we can arrive at
$$
\|\nabla\e_{\u,h}^j\|_{\Omega_j}^2+\|e_{p,h}^j\|_{\Omega_j}^2\lesssim
(1+chH^{1-d}\beta_d^{-1}(H))^{-{\cal
M}}(\|\nabla\hat\u_H^j\|_{\Omega}^2+\|\hat p_H^j\|^2_{\Omega}).
$$
Since ${\cal M}\approxeq\frac{H}{h}$, for small $H$, simple
calculation leads to the result of this lemma.
\end{proof}

\begin{remark}
{  {For small $H$}, due to a basic inequality $\lambda/2\le\ln(1+\lambda)\le \lambda$ for any $0\le\lambda\le 1$, we have
  \begin{eqnarray*}
    (1+chH^{1-d}\beta_d^{-1}(H))^{-{\cal
M}} &= &\exp\left(-{\cal
M}\ln(1+chH^{1-d}\beta_d^{-1}(H))\right)\\
&\approxeq&\exp\left(-{\cal
M}(chH^{1-d}\beta_d^{-1}(H))\right)\\
&\approxeq& \exp\left(\alpha_d |\ln H|\right)\\
&=&H^{\alpha_d},
  \end{eqnarray*}
with
  \begin{eqnarray*}
  \alpha_d=\frac{c{\cal
M}hH^{1-d}\beta_d^{-1}(H)}{|\ln H|}=
\left\{\begin{array}{ll}
({\cal
M}\frac{h}{H})\frac{c}{|\ln H|^2}\approxeq\frac{c}{|\ln H|^2}, \quad & d=2,\\
({\cal
M}\frac{h}{H})\frac{c}{|\ln H|}\approxeq\frac{c}{|\ln H|}, & d=3.
\end{array}\right.
   \end{eqnarray*}
Since  $\beta_d(H)=H^{-1}$ when $d=3$ and $\beta_d(H)=|\ln H|$ when
$d=2$ (defined in Lemma \ref{LEH01}), we can easily check that $chH^{1-d}\beta_d^{-1}(H)\le 1$ holds when $h\le H/c$ for $d=3$ or $h\le H|\ln H|/c$ for $d=2$, which is a very trivial condition.}

\end{remark}

Similar with the  estimate of $\|\nabla
\e_{\u,h}^j\|_{\Omega}+\|e_{p,h}^j\|_\Omega$ in Lemma \ref{LE2}, we can derive { an}
 estimate for the { approximate "local" error of the coarse mesh Galerkin approximation}, say
$\|\nabla\hat\u_H^j\|_{\Omega}+\|\hat p_H^j\|_{\Omega}$ in the following.
\begin{lemma}\label{LE4} {\it Under the assumptions of A1, A2 and A3,   we have for $j=1,2,\cdots,N$
\begin{eqnarray}\label{L45-1}
\|\nabla\hat\u_H^j\|_{\Omega}+\|\hat p_H^j\|_\Omega\lesssim \|\nabla(\u-\u_H)\|_{D_j}+\|p-p_H\|_{D_j}.
\end{eqnarray}
}
\end{lemma}

\begin{proof}
First  we rewrite the problem (\ref{errequh2}) as
\begin{eqnarray*}
&&B([\hat\u_H^j,\hat p_H^j],[\vv_h,q_h])=(\f,\phi_j\vv_h)-a(\u_H,\phi_j\vv_h)\\
&&\qquad-d(p_H,\phi_j\vv_h)+d(\phi_jq_h,\u_H),\quad\forall
[\vv_h,q_h]\in\X_0^h\times Q_h,
\end{eqnarray*}
or equivalently,
$$
B([\hat\u_H^j,\hat
p_H^j],[\vv_h,q_h])=a(\u-\u_H,\phi_j\vv_h)+d(p-p_H,\phi_j\vv_h)-d(\phi_jq_h,\u-\u_H).
$$
Since $\X_0^h\times Q^h$ is a stable pair of the finite element spaces, we have
\begin{eqnarray*}
    &&\|\nabla\hat\u_H^j\|_\Omega+\|\hat p_H^j\|_\Omega\lesssim \sup\limits_{[\vv_h,q_h]\in\X_0^h\times M_h}\frac{B([\hat\u_H^j,\hat p_H^j],[\vv_h,q_h])}{\|\nabla\vv_h\|_\Omega+\|q_h\|_\Omega}\\
    &&\qquad=\sup\limits_{[\vv_h,q_h]\in\X_0^h\times M_h}\frac{a(\u-\u_H,\phi_j\vv_h)+d(p-p_H,\phi_j\vv_h)-d(\phi_jq_h,\u-\u_H)}{\|\nabla \vv_h\|_\Omega+\|q_h\|_\Omega}\\
    &&\qquad\lesssim \sup\limits_{\vv_h\in\X_0^h}\frac{a(\u-\u_H,\phi_j\vv_h)}{\|\nabla \vv_h\|_\Omega}+\sup\limits_{\vv_h\in\X_0^h}\frac{d(p-p_H,\phi_j\vv_h)}{\|\nabla \vv_h\|_\Omega}
    +\sup\limits_{q_h\in M^h}\frac{d(\phi_jq_h,\u-\u_H)}{\|q_h\|_\Omega}.
\end{eqnarray*}
 Note that {   $I_H$ is the Lagrange finite element interpolation operator of $\X$ onto $\X^H$ and $\hat I_H=I-I_H$, especially, $\mbox{supp }\phi_j= D_j$, and  $a(\u-\u_H,\phi_j\vv_h)+d(p-p_H,\phi_j\vv_h)=a(\u-\u_H,\hat I_H[\phi_j\vv_h])+d(p-p_H,\hat I_H\phi_j\vv_h)$,} it's easy to derive
\begin{eqnarray*}
    &&a(\u-\u_H,\hat I_H[\phi_j\vv_h])=a(\u-\u_H,\hat I_H[\phi_jI_H\vv_h]+\hat I_H[\phi_j\hat I_H\vv_h])\\
    &&\qquad\lesssim \|\nabla(\u-\u_H)\|_{D_j}(\|\nabla[\hat I_H(\phi_jI_H\vv_h)]\|_{D_j}+\|\nabla \hat I_H[\phi_j\hat I_H\vv_h]\|_{D_j}),\\
    && d(p-p_H,\hat I_H\phi_j\vv_h) =d(p-p_H,\hat I_H[\phi_j I_H\vv_h]+\hat I_H[\phi_j\hat I_H\vv_h])\\
    &&\qquad\lesssim \|p-p_H\|_{D_j}(\|\nabla[\hat I_H(\phi_jI_H\vv_h)]\|_{D_j}+\|\nabla \hat I_H[\phi_j\hat I_H\vv_h]\|_{D_j}),\\
    &&d(\phi_jq_h,\u-\u_H)\lesssim \|q_h\|_{D_j}\|\nabla(\u-\u_H)\|_{D_j}.
\end{eqnarray*}
In view of  {\bf A1}, {\bf A2} and noting that $\phi_j$ is a linear
function on each mesh of  $\tau_\Omega^H$ with $|D\phi_j|\lesssim H^{-1}$,
\begin{eqnarray*}
&&\|\nabla[\hat I_H(\phi_jI_H\vv_h)]\|_{D_j}=(\sum\limits_{\tau_\Omega^H\subset D_j}\|\nabla\hat I_H(\phi_jI_H\vv_h)\|_{\tau_\Omega^H}^2)^\frac12\leq H(\sum\limits_{\tau_\Omega^H\subset D_j}\|D^2(\phi_jI_H\vv_h)\|_{\tau_\Omega^H}^2)^\frac12\\
&&\qquad\lesssim H(\sum\limits_{\tau_\Omega^H\subset D_j}[\|\phi_jD^2(I_H\vv_h)\|_{\tau_\Omega^H}^2+\|D\phi_jD(I_H\vv_h)\|_{\tau_\Omega^H}^2])^\frac12\\
&&\qquad\lesssim H(\sum\limits_{\tau_\Omega^H\subset D_j}[H^{-2}\|D(I_H\vv_h)\|_{\tau_\Omega^H}^2+H^{-2}\|D(I_H\vv_h)\|_{\tau_\Omega^H}^2])^\frac12\lesssim \|\nabla\vv_h\|_{D_j},\\
&&\|\nabla \hat I_H[\phi_j\hat I_H\vv_h]\|_{D_j}\lesssim\|D(\phi_j\hat I_H\vv_h)\|_{D_j}
\lesssim \|D\phi_j\hat I_H\vv_h\|_{D_j}+\|\phi_jD(\hat
I_H\vv_h)\|_{D_j}\lesssim \|\nabla \vv_h\|_{D_j}.
\end{eqnarray*}
Combining these estimates immediately yield
$$
\|\nabla\hat\u_H^j\|_{\Omega}+\|\hat p_H^j\|_\Omega\lesssim
\|\nabla(\u-\u_H)\|_{D_j}+\|p-p_H\|_{D_j}.
$$
The proof is complete.
\end{proof}

Now we can state the main theorem concerning the approximate accuracy of solutions derived by the scheme
(\ref{alerrequ})-(\ref{final}) as follows.
\begin{theorem}\label{theorem1} {\it Assume that  A1, A2, A3 and (\ref{ASGM}) hold and $[\u,p]\in\H^{r+1}(\Omega)\times\H^r(\Omega)$. Then we have
{

\begin{eqnarray}
\|\nabla(\u_h-\u_H^h)\|_{\Omega}+\|p_h-p_H^h\|_\Omega&&\lesssim
H^{\alpha_d}(\|\nabla(\u_h-\u_H)\|_{\Omega}+\|p_h-p_H\|_\Omega),\label{T46-11}\\
\|\u_h-\u_H^h\|_{\Omega}&&\lesssim
H(\|\nabla(\u_h-\u_H^h)\|_{\Omega}+\|p_h-p_H^h\|_\Omega),\label{T46-12}\\
\|\nabla(\u-\u_H^h)\|_{\Omega}+\|p-p_H^h\|_\Omega&&\lesssim
h^r+H^{\alpha_d}(\|\nabla(\u-\u_H)\|_{\Omega}+\|p-p_H\|_\Omega),\label{T46-21}\\
\|\u-\u_H^h\|_{\Omega}&&\lesssim
h^{r+1}+H(\|\nabla(\u-\u_H^h)\|_{\Omega}+\|p-p_H^h\|_\Omega),\label{T46-22}
\end{eqnarray}
}
where $\alpha_d>0$ is defined in Lemma \ref{LE2}. }
\end{theorem}

\proof {
We know from (\ref{post2}) that
$$
\|\nabla\E_\u^H\|_\Omega+\|E_p^H\|_\Omega\lesssim
\|\nabla(\u_h-\u_{H,h})\|_\Omega+\|p_h-p_{H,h}\|_\Omega.
$$
Since
$$
[\u-\u_H^h,p-p_H^h]=[\u-\u_h,p-p_h]+[\u_h-\u_{H,h},p_h-p_{H,h}]+[\E_\u^H,E_p^H],
$$
and
$$
\|\nabla(\u-\u_h)\|_\Omega+\|p-p_h\|_\Omega\lesssim h^r,
$$
the conclusions (\ref{T46-11}) and (\ref{T46-21}) follows immediately due to Lemma \ref{LE2}
and Lemma \ref{LE4} and the triangle inequality.}

To estimate the $L^2$-norms of the velocity,  we firstly introduce a Stokes projection
$P_H=[P_H^\u,P_H^p]$ defined from $\X_0\times Q$ onto $\X_0^H\times Q^H$:
for given $[\w,r]\in\X_0\times Q$, find $[P_H^\u\w,P_H^p
r]\in\X_0^H\times Q^H$ such that
$$
B([\vv_H,q_H],[\w-P_H^\u\w,r-P_H^p r])=0,\quad\forall
[\vv_H,q_H]\in\X_0^H\times Q^H.
$$

In view of the definitions of $\hat\u_{H,h}$ and $\hat p_{H,h}$, and
$\sum\limits_{j=1}^N\phi_j=1$ in $\Omega$, adding up all the
equations of (\ref{fictitious}) over $j$ with $\mu=0$ gives
\begin{eqnarray}\label{1stcorrectionequation}
&&B([\hat\u_{H,h},\hat p_{H,h}],[\vv_h,q_h])=(\f,\vv_h)-B([\u_H,p_H],[\vv_h,q_h])\\
&&\nonumber\qquad-\sum\limits_{j=1}^N\int_{\Gamma_j}\xii^j\vv_hds,\quad\forall[\vv_h,q_h]\in
\X_0^h\times Q^h.
\end{eqnarray}
Since $[\u_{H,h},p_{H,h}]=[\u_H+\hat\u_{H,h},p_H+\hat
p_{H,h}]$, then we have
$$
B([\u_{H,h},p_{H,h}],[\vv_h,q_h])=(\f,\vv_h)-\sum\limits_{j=1}^N\int_{\Gamma_j}\xii^j\vv_hds,\quad\forall[\vv_h,q_h]\in
\X_0^h\times Q^h.
$$
On the another aspect, by the definition of the Stokes projection
$P_H$ above, we can rewrite  (\ref{post2}) associated with the coarse mesh
correction as
\begin{eqnarray*}
&&B([\E_\u^H,E_p^H],[\vv_h,q_h])=(\f,P_H^\u\vv_h)-B([\u_{H,h},p_{H,h}],[P_H^\u\vv_h,P_H^p
q_h]).
\end{eqnarray*}
Combining these two equations leads to
\begin{eqnarray*}
&&B([\u_H^h,p_H^h],[\vv_h,q_h])=(\f,\vv_h)+(\f,P_H^\u\vv_h)-B([\u_{H,h},p_{H,h}],[P_H^\u\vv_h,P_H^p q_h])\\
&&\qquad\quad-\sum\limits_{j=1}^N\int_{\Gamma_j}\xii^j\vv_hds,\quad\forall[\vv_h,q_h]\in\X_0^h\times
Q^h.
\end{eqnarray*}

Finally, we obtain $\forall[\vv_h,q_h]\in\X_0^h\times Q^h$
$$
B([\u_H^h,p_H^h],[\vv_h,q_h])=(\f,\vv_h)-\sum\limits_{j=1}^N\int_{\Gamma_j}\xii^j(I-P_H^\u)\vv_hds,
$$
and $\forall[\vv_h,q_h]\in\X_0^h\times Q^h$
\begin{equation}\label{errorequation}
B([\u_h-\u_H^h,p_h-p_H^h],[\vv_h,q_h])=\sum\limits_{j=1}^N\int_{\Gamma_j}\xii^j(I-P_H^\u)\vv_hds.
\end{equation}
From this error equation about $[\u_h-\u_H^h,p_h-p_H^h]$, and by virtue of
Lemma \ref{LE0}, \ref{LE2} and \ref{LE4}, we can easily get
$$
\|\nabla(\u_h-\u_H^h)\|_\Omega+\|p_h-p_H^h\|_\Omega\lesssim
H^{\alpha_d}(\|\nabla(\u-\u_H)\|_\Omega+\|p-p_H\|_\Omega).
$$
Then we can derive the first result by using the triangle
inequality.

Next, we also use the
Aubin-Nitsche duality argument to show $L^2-$error estimate of the velocity. By {\bf A3}, for
$\u_h-\u_H^h\in\L^2(\Omega)$, there exists
$[\w,r]\in\H^2(\Omega)\times H^1(\Omega)$ safistying
$$
B([\vv,q],[\w,r])=(\u_h-\u_H^h,\vv),\quad\forall
[\vv,q]\in\X_0(\Omega)\times Q,
$$
and
$$
\|\w\|_{2,\Omega}+\|r\|_{1,\Omega}\lesssim \|\u_h-\u_H^h\|_{\Omega}.
$$
If taking $[\vv,q]=[\u_h-\u_H^h,p_h-p_H^h]$ here and in view of
(\ref{errorequation}), we have
\begin{eqnarray*}
&&\|\u_h-\u_H^h\|_{\Omega}^2=B([\u_h-\u_H^h,p_h-p_H^h],[\w,r])\\
&&\qquad=B([\u_h-\u_H^h,p_h-p_H^h],[\w-P_H^\u\w,r-P_H^pr]).
\end{eqnarray*}
Then we can proceeds the estimate as
\begin{eqnarray*}
&&\|\u_h-\u_H^h\|_\Omega^2\lesssim (\|\nabla(\u_h-\u_H^h)\|_\Omega+\|p_h-p_H^h\|_\Omega)\|\nabla(I-P_H^\u)\w\|_\Omega\\
&&\qquad\quad+\|\nabla(\u_h-\u_H^h)\|_\Omega\|(I-P_H^p)r\|_\Omega\\
&&\quad\lesssim H(\|\nabla(\u_h-\u_H^h)\|_\Omega+\|p_h-p_H^h\|_\Omega)(\|\w\|_{2,\Omega}+\|r\|_{1,\Omega})\\
&&\quad\lesssim
H(\|\nabla(\u_h-\u_H^h)\|_\Omega+\|p_h-p_H^h\|_\Omega)\|\u_h-\u_H^h\|_\Omega,
\end{eqnarray*}
which immediately yields
$$
\|\u_h-\u_H^h\|_\Omega\lesssim
H(\|\nabla(\u_h-\u_H^h)\|_\Omega+\|p_h-p_H^h\|_\Omega).
$$
By using the triangle inequality again, we prove
$L^2-$error estimate for the velocity.
\endproof


\begin{remark}
By the results in Theorem \ref{theorem1} above, especially (\ref{T46-11}) and (\ref{T46-12}), we can clearly improve the convergence orders by introducing a
two-grid iteration. Actually, we can verify that all the  lemmas and theorems above still hold when
replacing $[\u_H,p_H]$ by $[\u_H^h,p_H^h]$.
This suggests us to introducing the
following two-grid iterative scheme with a iteration number of K.
\begin{equation}\label{KKK}
K=[\alpha_d^{-1}+0.5]=\left\{\begin{array}{ll}
O(|\ln H|^2),\quad & d=2,\\
O(|\ln H|), & d=3.
\end{array}\right.
\end{equation}
\end{remark}

\noindent{\bf{Local and parallel two-grid iterative scheme:}}
\begin{description}
\item[Step 0.] Setting $k=0$;  Deriving $[\u_H,p_H]$ by solving (\ref{SGM}) and denoting by $[\u_H^{0,h},p_H^{0,h}]=[\u_H,p_H]$;
\item[Step 1.] For $k\geq 0$, solving the equations (\ref{alerrequ}) with $[\u_H,p_H]=[\u_H^{k,h},p_H^{k,h}]$ to get $\{[\hat\u_{H,h}^j,\hat p_{H,h}^j]\}_{j=1}^N$ for each $j$, which are
denoted as $\{[\hat\u_{H,h}^{k+1,j},\hat p_{H,h}^{k+1,j}]\}_{j=1}^N$. Then we construct $[\u_{H,h}^{k+1},p_{H,h}^{k+1}]$ by formula (\ref{post1});
\item[Step 2.] Deriving $[\E_\u^{k+1},E_p^{k+1}]$ by
solving (\ref{post2}) with $[\u_{H,h},p_{H,h}]=[\u_{H,h}^{k+1},p_{H,h}^{k+1}]$, and constructing
$$
[\u_H^{k+1,h},p_H^{k+1,h}]=[\u_{H,h}^{k+1}+\E_\u^{k+1},p_{H,h}^{k+1}+E_p^{k+1}].
$$
\item[Step 3.]Checking whether $k+1>K$, if yes, terminating the iteration and deriving $[\u_H^h,p_H^h]=[\u_H^{k+1,h},p_H^{k+1,h}]$; otherwise, letting $k:=k+1$ and turning to Step 1.
\end{description}

\begin{theorem}\label{Theorem2} {\it Suppose $[\u,p]\in\H^{r+1}(\Omega)\times H^r(\Omega)$, the final approximation $[\u_H^h,p_H^h]$ for the iterative scheme with $K$ defined by (\ref{KKK}) has the following error bounds}
\begin{eqnarray}
\label{ErrorH1} && \|\nabla(\u-\u_H^h)\|_{\Omega}+\|p-p_H^h\|_\Omega\lesssim h^r+H^{r+1},\\
\label{ErrorL2} && \|\u-\u_H^h\|_{\Omega}\lesssim h^{r+1}+H^{r+2}.
\end{eqnarray}
\end{theorem}
\begin{proof}
{  Note that, the Ritz-Projection of $[\u_H^{h},p_H^{h}]$ onto the corresponding course grid finite element spaces is nothing but the standard Galerkin approximation,
see (\ref{errorequation}), when selecting $[\vv_h,q_h]=[\vv_H,q_H]$,  we have
\begin{eqnarray*}
B([\u_h-\u_H^h,p_h-p_H^h],[\vv_H,q_H])=0,
\end{eqnarray*}
then,
\begin{eqnarray*}
B([\u_H^h,p_H^h],[\vv_H,q_H])=B([\u_h,p_h],[\vv_H,q_H])=(\f,\vv_H)=B([\u_H,p_H],[\vv_H,q_H]).
\end{eqnarray*}
Then,  we can update $[\u_H^{0,h},p_H^{0,h}]$ by $[\u_H^{1,h},p_H^{1,h}]$,
and execute  the previous local and parallel two-grid iterative scheme (without Step 0) again to get $[\u_H^{2,h},p_H^{2,h}]$, then follow (\ref{T46-11}) and (\ref{T46-12}) to derive the error estimates of the form,  equivalently, there exists a mesh independent constant
$c>0$,
\begin{eqnarray*}
\|\nabla(\u_h-\u_H^{2,h})\|_{\Omega}+\|p_h-p_H^{2,h}\|_\Omega&& \lesssim
H^{\alpha_d}(\|\nabla(\u_h-\u_H^{1,h})\|_{\Omega}+\|p_h-p_H^{1,h}\|_\Omega)\\
&&\lesssim
H^{2\alpha_d}(\|\nabla(\u_h-\u_H^{0,h})\|_{\Omega}+\|p_h-p_H^{0,h}\|_\Omega),\\
\|\u_h-\u_H^{2,h}\|_{\Omega}&&\lesssim
H(\|\nabla(\u_h-\u_H^{2,h})\|_{\Omega}+\|p_h-p_H^{2,h}\|_\Omega).
\end{eqnarray*}
Therefore, cyclically executing our  two-grid iterative scheme with updating data by a iteration of $K$ times, we can finally derive
\begin{eqnarray*}
\|\nabla(\u_h-\u_H^{K,h})\|_{\Omega}+\|p_h-p_H^{K,h}\|_\Omega&& \lesssim
H^{K\alpha_d}(\|\nabla(\u_h-\u_H^{0,h})\|_{\Omega}+\|p_h-p_H^{0,h}\|_\Omega)\\
&&\lesssim H^{K\alpha_d}(\|\nabla(\u_h-\u_H)\|_{\Omega}+\|p_h-p_H\|_\Omega),\\
\|\u_h-\u_H^{K,h}\|_{\Omega}&&\lesssim
H(\|\nabla(\u_h-\u_H^{K,h})\|_{\Omega}+\|p_h-p_H^{K,h}\|_\Omega).
\end{eqnarray*}

By the choice of $K$ within (\ref{KKK}) and the triangle inequality, since $[\u_H^h,p_H^h]=[\u_H^{K,h},p_H^{K,h}]$ now, we finally get (\ref{ErrorH1}) and (\ref{ErrorL2}).
}

\end{proof}

\begin{remark}
By the estimates in Theorem  \ref{Theorem2}, we can clearly know that in order to get the optimal  convergence orders of the velocity - pressure pair in  the  sense of $\H^1-L^2$-norms, or that of  the velocity in $\L^2$-norm, we should choose $H$ and $h$ such that
\begin{equation}\label{CHh}
h\sim H^\frac{r+1}{r}\quad\mbox{or}\quad h\sim H^\frac{r+2}{r+1},
\end{equation}
respectively.
\end{remark}

\section{Numerical Tests}\label{Tests}

In all the numerical experiments below, the algorithms are implemented using
public domain finite element software Freefem++\cite{free}. All
simulations were performed on  a Dawning parallel cluster composed
of 64 nodes (each node consists of eight-core 2.0 GHz CPU,
 8$\times$2 GB DRAM, and all the nodes are connected via 20Gbps
InfiniBand). The message-passing interface is supported by MPICH.

We introduce the following
notations for convenience:

SFEM means   the standard finite element method.

EPLP  denotes the expandable local and parallel two-grid finite
element iterative scheme.

Wall time   covers the maximal CPU time among all processors used
for EPLP, including the CPU time for solving the two global coarse
grid problems and the parallel time for solving all the subproblems.

\subsection{Problem 1}

Firstly, to verify the theoretical  results, we consider the following
2D numerical example (referred as Problem 1) with exact solution
 \begin{eqnarray*}
 &&\mathbf u=(10x^2(x-1)^2y(y-1)(2y-1),-10x(x-1)(2x-1)y^2(y-1)^2),\\
&&p=10(2x-1)(2y-1).
 \end{eqnarray*}
 The domain is selected as the unit square
$\Omega=[0,1]\times [0,1]$ with a uniform triangulation $T^H$.  The Taylor-Hood finite element pairs are used  in solving the Stokes equation.
\begin{figure}[htbp]
\begin{center}
\includegraphics[width=50mm,height=40mm]{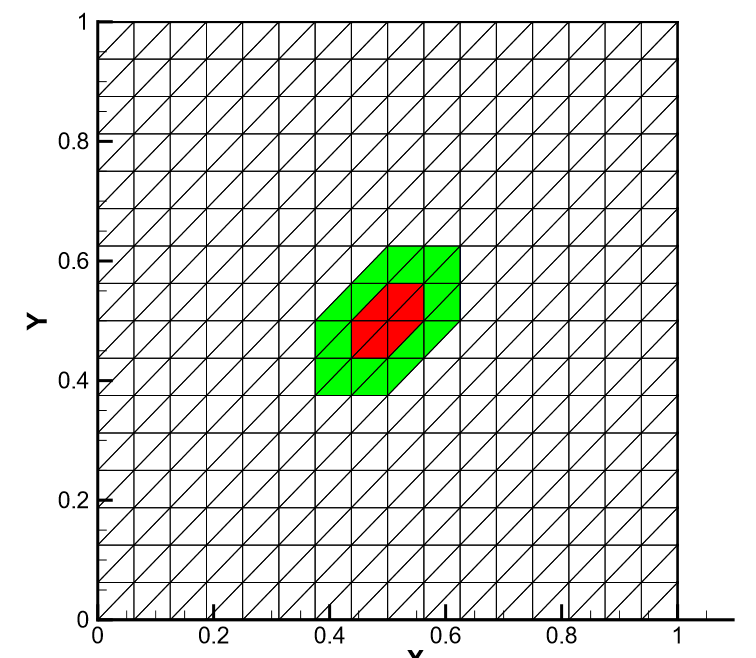}\\
\caption{  Structured mesh for Problem 1, $D_j$:  local domain (red) and $\Omega_j$:  expanded domain of $D_j$ with one mesh layer (both red and green). \label{figstruct}}
\end{center}
\end{figure}

According to   Theorem  \ref{Theorem2} for EPLP, we
know that the estimate has  the form of
\begin{eqnarray*}
&&\|\nabla(\mathbf{u}-{\mathbf{u}}_H^h)\|_{\Omega}
+\|p-{p}^h_H\|_{\Omega}=O(h^2+ H^3),
\end{eqnarray*}
 which suggests us to select $H$ and $h$ such that
$h\approxeq H^{\frac32}$ when verifying the optimal convergence orders for $H^1$-norm of velocity
and $L^2$-norm of pressure.

To this end, for deriving the approximate $H^1$-error of velocity and $L^2$-error of pressure,  we
implement EPLP in parallel  with coarse meshes of the
sizes H=1/16,\ 1/25,\ 1/36,\ 1/49,\ 1/64 and corresponding fine meshes with the
sizes $h=1/64,\ 1/125,\ 1/216,\ 1/343, \ 1/512$, under the parallel environment of fixed 64 processors. Meanwhile, SFEM is executed
for the same fine mesh on a single processor. The
results for both methods are presented in Table  \ref{H1LPPUM}, \ref{H1SFEM},
respectively.

\begin{table}[!h]
\tabcolsep 0pt \caption{The  errors of SFEM for Problem 1.
{\label{H1SFEM}}} \vspace*{-10pt}
\begin{center}
\def\temptablewidth{0.9\textwidth}
{\rule{\temptablewidth}{1pt}}
\begin{tabular*}{\temptablewidth}{@{\extracolsep{\fill}}cccccc }
\hline   $h$&$||\mathbf u-\mathbf u_{h}||_{1,\Omega}$ &Order & $||p-p_h||_{0,\Omega}$ &  Order   &CPU  \\
 \hline $1/64$&  0.000205741&        &  0.000630369
 &  &  2.718      \\
        $1/125$& 5.396e-05& 1.99927& 0.000165247&2.00000 &   68.49\\
             $1/216$ &- & - &-& - & - \\
              $1/343$ &- & - &-& - & - \\
               $1/512$ &- & - &-& - & - \\
  \hline
\end{tabular*}
\end{center}
 \end{table}

\begin{table}[!h]
\tabcolsep 0pt \caption{The errors of EPLP for Problem 1,   $h
\approxeq H^{\frac32}$, 64 processors. {\label{H1LPPUM}}}
\vspace*{-10pt}
\begin{center}
\def\temptablewidth{0.9\textwidth}
{\rule{\temptablewidth}{1pt}}
\begin{tabular*}{\temptablewidth}{@{\extracolsep{\fill}}ccccccccc }
         \hline  $H$& $h$&$K$&$||\mathbf u-{\mathbf u}_H^{h}||_{1,\Omega}$ &Order & $||p-{p}_H^h||_{0,\Omega}$ &  Order   & Wall time  \\
 \hline  $1/16$&$1/64$&2 &0.000299829&         & 0.00077599&  &  10.54\\
 $1/25$&$1/125$&3 &5.44032e-05&     2.54961    & 0.000165716&  2.30623&  91.51
\\
       $1/36$& $1/216$&3&1.81385e-05& 2.00815 &5.53793e-05
&  2.00391 &   465.08\\
       $1/49$& $1/343$&3&7.19776e-06& 1.99862 &2.19737e-05& 1.99882 &   2084.38\\
            $1/64$& $1/512$&4 &3.25815e-06&1.97859  &1.01485e-05&1.92844&  7892.11\\
  \hline
\end{tabular*}
\end{center}
 \end{table}

\begin{table}[!h]
\tabcolsep 0pt \caption{The $L^2$-error of Velocity by EPLP for
Problem 1, $h\approxeq H^{\frac43}$, 64 processors.
{\label{L2LPPUM}}} \vspace*{-10pt}
\begin{center}
\def\temptablewidth{0.9\textwidth}
{\rule{\temptablewidth}{1pt}}
\begin{tabular*}{\temptablewidth}{@{\extracolsep{\fill}}cccccccc }
         \hline  $H$& $h$&$K$&$||\mathbf u-{\mathbf u}_H^{h}||_{0,\Omega}$ &Order   & Wall time  \\
 \hline  $1/32$&$1/96$&3 &1.0881e-07  &          &  80.44
\\
 $1/64$&$1/256$&4 &5.71659e-09  &    3.00382  &  2032.04
\\
       $1/96$& $1/384$&4&1.74889e-09  & 2.92107   &  10435.2\\
  \hline
\end{tabular*}
\end{center}
 \end{table}

From Table  \ref{H1LPPUM}, the optimal orders for the $H^1$-error of velocity
and $L^2$-error of pressure  by EPLP are observed, which can verify the  theoretical result. Also
the approximate accuracy by EPLP are comparable
 with those of SFEM under same fine meshes in Table \ref{H1SFEM}.  Noting that, in the present
computational environment, SFEM fails for mesh with$h=1/216$ or finer, meanwhile our EPLP still works well  for much finer meshes including $h=1/512$.

Similarly, by  Theorem  \ref{Theorem2}, we choose mesh pairs with  a relation of $h\approxeq
H^{\frac43}$ to check the order for $L^2$-error of the velocity computed by our EPLP. The parameters and computational
results  are shown in Table \ref{L2LPPUM}, from which the optimal order of $O(h^3)$ is
verified.  This also supports the result in the theoretical analysis.

\begin{table}[!h]
\tabcolsep 0pt \caption{Wall time $T(J)$ in seconds, speedup $S_p$
and parallel efficiency $E_p$  for Example 1, with $H=1/36$.
{\label{P2019}}} \vspace*{-10pt}
\begin{center}
\def\temptablewidth{0.9\textwidth}
{\rule{\temptablewidth}{1pt}}
\begin{tabular*}{\temptablewidth}{@{\extracolsep{\fill}}ccccccc }
\hline   $J$  & 2 & 4 & 8 & 16 &32 &64  \\
 \hline
  $T(J)$        &  9204.77
&4684.47&2520.25&  1392.31&747.51&465.08\\
  $S_p=\frac{T(2)}{T(J)}$               &1.00   & 1.96 & 3.65 & 6.61&  12.31& 19.79 \\
  $E_p=\frac{2\times T(2)}{J\times T(J)}$&1.00     &0.98& 0.91 & 0.83&0.77&0.62\\
  \hline
\end{tabular*}
\end{center}
 \end{table}

The performance of a parallel algorithm in a homogeneous parallel
environment is usually measured by speedup and parallel efficiency,
commonly defined as
\begin{eqnarray*}
S_p=\frac{T(J_1)}{T(J_2)},\ E_p=\frac{J_1\times T(J_1)}{J_2\times
T(J_2)},
\end{eqnarray*}
where $T(J_1)$ and $T(J_2)$ ($J_1\le J_2$) are  wall times of the
parallel program when using $J_1$ and $J_2$ processors, respectively.

Table  \ref{P2019} reports the wall time of EPLP in a parallel
environment using processors of number $J=2,\ 4,\ 8,\ 16,\
32,\ 64$, and presents the corresponding speedup and parallel
efficiency, which are computed by comparison with $J_1=2$. These
results show  good parallel performance of our EPLP.

\subsection{Problem 2}

\begin{figure}[htbp]
\begin{center}
\includegraphics[width=50mm,height=40mm]{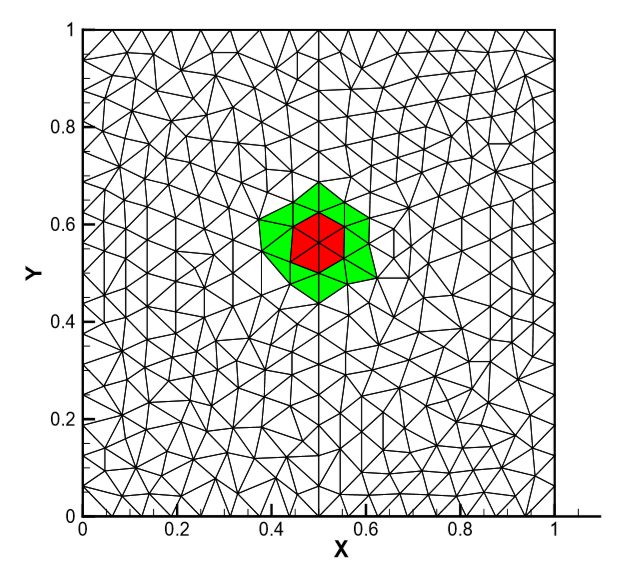}\\
 \caption{  A regular unstructured mesh for Problem 2, $D_j$: a  local domain (red) and $\Omega_j$: an expanded domain of $D_j$ (both red and green). \label{figunstruct}}
\end{center}
\end{figure}

\begin{table}[!h]
\tabcolsep 0pt \caption{The errors of EPLP for Problem 2  with
Taylor-Hood elements, $h\approxeq H^{\frac32}$, with 64
processors. {\label{E2P2P1}}} \vspace*{-10pt}
\begin{center}
\def\temptablewidth{0.9\textwidth}
{\rule{\temptablewidth}{1pt}}
\begin{tabular*}{\temptablewidth}{@{\extracolsep{\fill}}ccccccccc }
         \hline  $H$& $h$&$K$&$||\mathbf u-{\mathbf u}^{h}||_{1,\Omega}$ &Order & $||p-{p}^h||_{0,\Omega}$ &  Order   & Wall time  \\
 \hline  $1/16$&$1/64$&2 &0.00237151&         &9.76556e-05&  &  10.75
\\
 $1/32$&$1/160$&3 &0.000415357&     1.90130    & 1.51911e-05& 2.03073&  196.61
\\
       $1/48$& $1/288$&3&0.000137697&1.87837  &4.88635e-06
&   1.92972&    1137.06
\\
       $1/64$& $1/448$&4&4.91655e-05& 2.33089 &2.05054e-06&  1.96532 &   6254.84\\
  \hline
\end{tabular*}
\end{center}
 \end{table}

\begin{table}[!h]
\tabcolsep 0pt \caption{The errors of EPLP for Problem 2 with
Mini-elements, $h\approxeq H^{2}$, with 64 processors.
{\label{E2P1bP1}}} \vspace*{-10pt}
\begin{center}
\def\temptablewidth{0.9\textwidth}
{\rule{\temptablewidth}{1pt}}
\begin{tabular*}{\temptablewidth}{@{\extracolsep{\fill}}ccccccccc }
         \hline  $H$& $h$&$K$&$||\mathbf u-{\mathbf u}^{h}||_{1,\Omega}$ &Order & $||p-{p}^h||_{0,\Omega}$ &  Order   & Wall time  \\
 \hline  $1/8$&$1/64$&2 &0.151405&         &0.0550405&  &  14.67
\\
 $1/16$&$1/256$&2 &0.0361013&   1.03415     & 0.0131228& 1.03421 &  538.43
\\
       $1/32$& $1/1024$&3&    0.00924041   &  0.983011     &  0.00212624   &    1.31285     &52038.4\\
  \hline
\end{tabular*}
\end{center}
 \end{table}

To further test our EPLP, we also consider another smooth problem
(referred as Problem 2) with exact solution of the trigonometric form
 \begin{eqnarray*}
 &&\mathbf u=(\sin(\pi x)^2 \sin(2 \pi y),-\sin(2 \pi x) \sin(\pi y)^2),\\
&&p=\cos(\pi x) \cos(\pi y),
 \end{eqnarray*}
 on the  domain $\Omega=[0,1]\times [0,1]$. The regular unstructured triangulations, and
Taylor-Hood elements  and Mini-elements are used for our EPLP.
The corresponding approximate results are listed in
Table \ref{E2P2P1} and \ref{E2P1bP1}, respectively.
From these two tables, one can observe that,  for EPLP with
$K$ iterations,  the $H^1$-error of velocity and $L^2$-error of
pressure  achieve the optimal orders, namely, $O(h^2)$ by Taylor-Hood
elements, and $O(h)$ by Mini-elements, which support theoretical results in  Theorem  \ref{Theorem2}.

\begin{table}[!h]
\tabcolsep 0pt \caption{The errors of EPLP  for Problem 3 with
Taylor-Hood element pair, $h\approxeq H^{\frac32}$, with 64 processors.
{\label{E3P2P1}}} \vspace*{-10pt}
\begin{center}
\def\temptablewidth{0.9\textwidth}
{\rule{\temptablewidth}{1pt}}
\begin{tabular*}{\temptablewidth}{@{\extracolsep{\fill}}cccccccc}
         \hline  $H$& $h$&$K$&$||\mathbf u-{\mathbf u}^{h}||_{1,\Omega}$ &Order & $||p-{p}^h||_{0,\Omega}$ &  Order    \\
 \hline  $1/4$&$1/8$&1 &0.763&         &0.0408459&
\\
 $1/8$&$1/24$&1 &0.104595   & 1.80879    & 0.00564176&1.80192
\\
       $1/16$& $1/64$&2&    0.0154352  & 1.95084    &  0.000769727  &   2.03085    \\
  \hline
\end{tabular*}
\end{center}
 \end{table}

\subsection{Problem 3}

As the final experiment, we will test our EPLP using a three dimensional problem
(referred as Problem 3) with  $\nu=0.05$ and exact solution
 \begin{eqnarray*}
 &&\mathbf u=\left(\sin(\pi x)^2 \sin(2 \pi y) \sin(2 \pi z),-\sin(2 \pi x) \sin(\pi y)^2\sin(2 \pi z),\sin(2\pi x) \sin(2\pi y)\right),\\
&&p=\cos(2\pi x) \cos(2\pi y)\cos(2\pi z).
 \end{eqnarray*}
The uniform  triangulation and Taylor-Hood elements are used for $T^H$ with mesh size $H$ and the computational domain is
chosen as the unit cube $\Omega=[0,1]\times [0,1]\times [0,1]$.  The computational results are shown in Table
\ref{E3P2P1},  from which one can observe that, while $h$ deceases,   by suitable
iteration of  $K$, both the $H^1$-error of velocity
and $L^2$-error of pressure can reach the optimal orders of
$O(h^2)$, which also support our theoretical analysis.

\section{Conclusions}\label{Conclusions}
In this paper, we have designed an expandable local and parallel  two-grid finite element iterative scheme based on superposition
principle for the Stokes problem.  The optimal convergence orders of the scheme are analyzed and obtained within suitable two-grid
iterations while  numerical tests in  2D and 3D are carried out to show the flexible and high
efficiency of  the scheme.  The extension of the scheme  to the time-dependent problems or nonlinear problems, e.g.,
Navier-Stokes equations, will be our further work.

\end{document}